\documentclass[12pt
]{amsart}
\usepackage[margin=1in]{geometry}

\usepackage{mathtools}
\mathtoolsset{showonlyrefs}

\usepackage{amsmath}
\usepackage{amsfonts}
\usepackage{amssymb}
\usepackage{amsthm}
\usepackage{setspace}
\usepackage{amsrefs}

\usepackage{hyperref}
\usepackage{xcolor}
\hypersetup{
    colorlinks,
    linkcolor={black},
    citecolor={blue!50!black},
    urlcolor={blue!80!black}
}

\title[]{Global Newlander-Nirenberg theorem for domains with $C^2$ boundary}

\date{\today}

\author[C. Gan]{Chun Gan}
\author[X. Gong]{Xianghong Gong$^\dag$}

 \address{Department of Mathematics,
 University of Wisconsin-Madison, Madison, WI 53706, U.S.A.}
 \email{chun@math.wisc.edu}
  \email{gong@math.wisc.edu}

\thanks{${\dag}$ Partially supported by a grant from the Simons
Foundation (award number:~505027)}



\newcommand{\dist}{\operatorname{dist}}

\newtheorem{thm}{Theorem}[section]
\newtheorem{cor}[thm]{Corollary}
\newtheorem{prop}[thm]{Proposition}
\newtheorem{lemma}[thm]{Lemma}

\theoremstyle{definition}

\newtheorem{rem}[thm]{Remark}

\renewcommand{\th}[1]{\begin{thm}\label{#1}}
\renewcommand{\eth}{\end{thm}}
\newcommand{\co}[1]{\begin{cor}\label{#1}}
\newcommand{\eco}{\end{cor}}
\renewcommand{\le}[1]{\begin{lemma}\label{#1}}
\newcommand{\ele}{\end{lemma}}
\newcommand{\pr}[1]{\begin{prop}\label{#1}}
\newcommand{\epr}{\end{prop}}

\newcommand{\ga}{\begin{gather}}
\newcommand{\ega}{\end{gather}}
\newcommand{\gan}{\begin{gather*}}
\newcommand{\egan}{\end{gather*}}
\newcommand{\al}{\begin{align}}
\newcommand{\eal}{\end{align}}
\newcommand{\aln}{\begin{align*}}
\newcommand{\ealn}{\end{align*}}
\newcommand{\eq}[1]{\begin{equation}\label{#1}}
\newcommand{\eeq}{\end{equation}}

\newcommand{\f}[2]{\frac{#1}{#2}}

\newcommand{\ci}{~\cite}

\newcommand{\ov}{\overline}

\newcommand{\rr}{\mathbb R}
\newcommand{\Del}{\Delta}
\newcommand{\pd}{\partial}

\newcommand{\RE}{\operatorname{Re}}
\newcommand{\IM}{\operatorname{Im}}
\renewcommand{\dbar}{\overline\partial}

\newcommand{\rea}[1]{$(\ref{#1})$}
\newcommand{\rl}[1]{Lemma~\ref{#1}}

\newcommand{\rp}[1]{Proposition~\ref{#1}}
\newcommand{\rt}[1]{Theorem~\ref{#1}}

\newcommand{\rta}[1]{Theorem~$\ref{#1}$}

\newcommand{\db}{\dbar}

\newcounter{pp}
\newcommand{\bpp}{\begin{list}{$\hspace{-1em}\alph{pp})$}{\usecounter{pp}}}
\newcommand{\epp}{\end{list}}

\newcounter{ppp}
\newcommand{\bppp}{\begin{list}{$\hspace{-1em}(\roman{ppp})$}{\usecounter{ppp}}}
\newcommand{\eppp}{\end{list}}

\def\beq{\begin{equation}}
\def\eeq{\end{equation}}

\def\a{\alpha}
\def\b{\beta}

\def\l{\lambda}
\def\de{\delta}

\def\var{\varphi}

\def\s{\sigma}
\def\Z{\mathbb{Z}}
\def\e{\epsilon}
\def\th{\theta}
\def\g{\gamma}

\def\R{\mathbb{R}}

\def\C{\mathbb{C}}
\def\U{\mathcal{U}}
\def\Om{\Omega}

\def\N{\mathbb{N}}

\def\k{\kappa}

\def\o\{\O\}
\def\H{\mathcal{H}}
\def\BH{\mathbb{H}}
\def\L{\mathcal{L}}
\def\lap{\triangle}
\def\V{\mathcal{V}}
\def\om{\omega}
\def\ll{\left|}
\def\d{\partial}
\def\we{\wedge}
\def\lapla{\Delta}
\def\bo{\square}
\def\Div{\text{div}}
\def\grad{\triangledown}

\newcommand*{\defeq}{\mathrel{\vcenter{\baselineskip 0.5ex \lineskiplimit0pt
                     \hbox{\scriptsize.}\hbox{\scriptsize.}}}%
                     =}
\newcommand{\norm}[1]{\lVert#1\rVert}
\newcommand{\znorm}[1]{\lvert#1\rvert}

\newcommand{\pair}[1]{\left<#1\right>}

\begin{document}
\begin{abstract}
The Newlander-Nirenberg theorem says that a formally integrable complex structure is locally equivalent to the standard complex structure in the complex Euclidean space. In this paper, we consider two natural generalizations of the Newlander-Nirenberg theorem under the presence of a $C^2$ strictly pseudoconvex boundary. When a given formally integrable complex structure $X$ is defined on the closure of a bounded strictly pseudoconvex domain with $C^2$ boundary  $D \subset \mathbb C^n$,  we show the existence of global holomorphic coordinate systems defined on $\overline{D}$ that 
transform $X$ into the standard complex structure provided that $X$ is sufficiently close to the standard complex structure. Moreover, we show that  such closeness is stable under a small $C^2$ perturbation of $\partial D$. As a consequence, when a given formally integrable complex structure is  defined on a one-sided neighborhood of some point in a $C^2$ real hypersurface $M \subset \mathbb C^n$, we prove the existence of local one-sided holomorphic coordinate systems provided that $M$ is strictly pseudoconvex with respect to the given complex structure.  We also obtain results when the structures are finite smooth.
\end{abstract}

\maketitle

\section{Introduction}
\setcounter{thm}{0}\setcounter{equation}{0}

\
 
Given  a formally integrable 
smooth (i.e.\ $C^\infty$)  almost complex structure defined on $\ov D$, where $D$ is a bounded domain
in $\R^{2n}$, we consider the problem of finding smooth global holomorphic coordinate systems on $\ov D$ for the structure. By a smooth global holomorphic coordinate system on $\ov D$, we mean a smooth diffeomorphism sending $\ov D$ onto $\ov{ D'}$ where $D'$ is a domain in $\C^n$, while the diffeomorphism transforms the given complex structure into the standard complex structure on $\C^n$.
The classical Newlander-Nirenberg theorem asserts the existence of local holomorphic coordinate systems for a formally integrable almost complex structure defined near 
an interior point of $D$.
The main result of this paper is to show the existence of such global holomorphic coordinate systems 
when the structure is  a small perturbation of the standard complex structure on $\ov D$, 
where $D$ is a bounded, strictly pseudoconvex domain with $C^2$ boundary.
When both boundary and the  complex structure are $C^\infty$, this result is due to R.~Hamilton through a general program~\cites{MR0477158, MR594711, MR656198}.

We will use our global result 
 to show the existence of local holomorphic coordinate systems on a small,
one-sided neighborhood of a given point in a real hypersurface $M$: If the formally integrable almost complex structure is defined on $U\cup M$, where $U$ is a domain in $\C^n$ and $M$  is a piece of 
$C^2$ boundary of $U$ that is strictly pseudoconvex with respect to the given complex structure on $U\cup M$, then for any
boundary point $p\in M$ there is a smooth diffeomorphism, defined 
on
a neighborhood of $p$
 in $U\cup M$, that transforms the complex strucutre into the standard one.
 When both the real hypersurface and the almost complex structure are $C^\infty$, this result is due to Catlin  \cite{MR959270} and Hanges-Jacobowitz~\cite{MR980299}.

Therefore, by restricting to $\partial D\in C^2$, we establish 
 results
for the complex structures on strictly pseudoconvex domains with the {\it minimum} smoothness required to define the strict Levi-pseudoconvexity.
For simplicity, we shall  refer to the existence of global (resp.~local) holomorphic coordinate systems as a global (resp.~local) Newlander-Nirenberg theorem with ($C^2$)  boundary.

To state our results more precisely, we first recall some definitions.
Let $p\in \R^{2n}$ with $n\geq
2$. Let $X_{1},\cdots, X_{n}$ be vector fields  defined near $p$ and having  $C^1$ complex coefficients. 
We say that $\{X_{\a}\}_{\a=1}^n$ defines an almost complex structure near $p$ if
\eq{}
X_1,\cdots, X_{n}, \ov {X}_{ 1},\cdots, \ov{X}_{ n}
\,\,\text{are $\C$--linearly independent at $p$.}
\eeq
Here $\ov{ X_{\a}}$ denotes the complex conjugate of $X_{\a}$.  Let $[X_{\a},X_{\b}]=X_{\a} X_{\b} -X_{\b} X_{\a}$ be the Lie bracket of $X_{\a},X_{\b}$.  The almost complex structure is said to be
{\it formally integrable} if in addition there exist functions $a_{\a\b}^\gamma$  such that
\eq{integrability}
[X_{\a},X_{\b}]=a_{\a\b}^{\gamma} X_{ \gamma}
\eeq
near $p$ for $\a,\b=1,\cdots, n$. Here we have used Einstein convention to sum over the repeated index $\g$ and we shall adapt this convention throughout the paper.

Recall that a domain $D\subset \C^n$ with $C^2$ boundary is said to be  strictly pseudoconvex with respect to the standard complex structure at $p\in \d D$ if there exists some open neighborhood $U$ of $p$ and a $C^2$ real-valued function $\rho :  U\to \R$ such that the following hold:  $D\cap U =\{z\in U  : \rho< 0\}$, $\rho(p)=0$, $d \rho (p)\neq 0$ and $\sqrt{-1}\sum\frac{\d^2 \rho }{\d z_\a\d \overline z_\b}(p) t_\a t_{\bar \b}>0$ for all vectors $t \in \C^n\backslash\{0\}$ satisfying $\sum t_\a \frac{\d \rho}{\d z_\a}(p)=0$.  Finally, for $0< a <\infty$ let $\znorm{\ \cdot\ }_{D, a}$ be the H\"older-Zygmund norm of $\Lambda^a(\ov{D})$ for a domain $D\subset \R^n$. 
Note that $\Lambda^a$ is the H\"older space $C^a$ in equivalent norms when $a$ is non integer; see section \ref{sec Settings} for the definition of Zygmund spaces.

\medskip

Our main result is the following.
 
\begin{thm} \label{Global NN}
Let $5 < r \leq\infty$. Let $D$ be a domain in $\C^n$ with $C^2$ boundary that is strictly pseudoconvex with respect to  the standard complex structure on $\C^n$.  Let $X_{\overline\a}= \d_{\overline\a}+A_{\overline\a}^\b \d_\b$, $\a=1,\cdots, n$ be   $\Lambda^{r}(\overline{D})$ vector fields defining  a formally integrable almost complex structure on $\ov D$. 
There exist positive constants $\delta_r(D)$  and $\delta_5(D)$ such that if  
\ga\label{A_r}
\znorm{A}_{ D, r}<\de_r(D), \quad r < \infty, \\
\znorm{A}_{D, 5} < \de_5(D), \quad r = \infty,
\end{gather}
 then there exists an embedding $F$ of $\ov D$ into $ \C^n$ such that $dF(X_{\bar1}), \dots, dF(X_{\bar n})$ are in the span of $\f{\d}{\d \ov z_1},\dots, \f{\d}{\d \ov z_n}$
 while $F\in \Lambda^{r-1}(\overline{D})$ if $r<\infty$ and  $F\in C^\infty(\ov{D})$ if $r = \infty$. Moreover, the constants $\de_{r} (D), \de_5(D)$ depend only on the $C^2$ norm of a given defining function of domain $D$ and is lower stable under a small $C^2$ perturbation of $\d D$.
\end{thm}

In fact,  the condition \eqref{A_r} can be  relaxed. See Section \ref{convergence} for more detail. Notice that when the structure is smooth, we only need to control $\znorm{A}_{D,5}$ in order to achieve a smooth embedding.

\medskip

The {\emph {lower stability}} of $\de_r(D)$  in the theorem means that  for any domain $\widetilde D$ of which a defining function is sufficiently close to
a given defining function of $D$ in $C^2$ norm, we 
have
$$
\de_r(D) \leq 
C_r(D)\de_r(\widetilde D)
$$ for some constant $C_r(D)>0$ 
possibly dependent of $D$ 
but  
 independent of $\widetilde D$, while  the theorem 
 remains true for $\widetilde D$. This is an important ingredient in our proof of the local Newlander-Nirenberg theorem with $C^2$ boundary which we now describe.

Let $U$ be an open subset of $\C^{n}$ and let $M\subset \partial U$ be a real $C^2$ hypersurface in $\C^n$.
Let $X_{\bar 1},\dots, X_{\bar n}$ be $C^1$ vector fields on $U\cup M$ that define a formally integrable almost complex structure on $U\cup M$.
Let $T^{0,1}(U\cup M,X)$ be the  span of $\{ X_{\bar 1},\dots,  X_{\bar n}\}$ and $\Lambda^{0,1}(U\cup M,X)$ be its dual bundle.  An integrable almost complex structure $\{X_{\overline\a}\}_{\a=1}^n$ induces a natural decomposition of the exterior derivative  $d= \d_X+\dbar_X$. Here $\dbar _X  :  \Lambda_X^{p,q}\to \Lambda_X^{p,q+1}$,  $\dbar_X^2=0$, $\d_X$ is its conjugate,  and $\Lambda_X^{p,q}$ is the exterior algebra of smooth differential forms on $U\cup M$ of type $(p,q)$ w.r.t $\{X_{\overline\a}\}_{\a=1}^n$.

We say $M$ is strictly pseudoconvex w.r.t. $(U\cup M, \{X_{\overline\a}\}_{\a=1}^n)$,  if  for each $p\in M$, there exists a $C^2$ function $\rho$, defined in a neighorhood $\om$ of $p$ such that $\omega\cap U=\{z\in \omega:\rho(z)<0\}$, $\rho=0$ on $M\cap \om$, $d\rho(p)\ne 0 $ and
\eq{spc_X}
\sqrt{-1}\d_X\dbar_X \rho(p)(v, \overline v)>0, \quad \text{ $\forall v\in T_p^{1,0}(U\cup M,X)\cap(T_pM\otimes\C)$}, \quad v\neq0.
\eeq

We 
now can state the following local Newlander-Nirenberg theorem with boundary.

\begin{thm}
\label{Local NN} Let $5<r\leq\infty$.
Let $U$ be a domain in $\C^n$ whose boundary contains a piece of $C^2$  real hypersurface $M$ and let $X_{\bar 1},\dots, X_{\bar n}$ be $\Lambda^r( U \cup M )$ vector fields
defining a formally  integrable almost complex structure on $U \cup M$.
Assume that $M$ is strictly pseudoconvex with respect to $(U\cup M,\{X_{\overline\a}\}_{\a=1}^n)$.  Then for each $p\in M$, there exists a diffeomorphism $F$ defined on a neighborhood $\omega$ of $p$ in $U\cup M$ such that $dF(X_{\bar1}), \dots, dF(X_{\bar n})$ are in the span of $\f{\d}{\d \ov z_1},\dots, \f{\d}{\d \ov z_n}$  while $F\in \Lambda^{r-1}(\overline{\om})\cap \Lambda^{r+1}(\omega\cap U)$ and $F\in C^\infty(\ov\om)$ if $r=\infty$.
\end{thm}
For $p\in U$, this is the classical Newlander-Nirenberg theorem \cites{MR88770, MR149505, MR0208200, MR1045639, MR0253383, MR999729}.  
 Very recently, Street ~\ci{St2020} obtained a sharp result for the elliptic structures.
 By a result of Hill~\cite{MR1128593}, the local Newlander-Nirenberg theorem with  boundary can fail 
 for a suitable formal integrable smooth complex structure on a  domain of which the boundary is smooth and has one negative Levi-eigenvalue.

As mentioned above,  under the assumptions that both boundary and almost complex structure are $C^\infty$, the global Newlander-Nirenberg theorem  with boundary was first proved by Hamilton \cite{MR0477158} and  the local version was shown by Catlin \cite{MR959270}, Hanges-Jacobowitz \cite{MR980299} independently. In fact, Hamilton proved a more general version of Theorem \ref{Global NN} assuming that $D$ is a 
relatively compact subset with smooth boundary in a complex manifold $Y$ with $H^{1}(D,T^{1,0}D)=0$. Catlin proved a local Newlander-Nirenberg theorem with smooth pseudoconvex boundary. We note that these results are all carried out in $C^\infty$ category with $\partial D\in C^\infty$, using $\dbar$-Neumann-type methods.

To prove \rt{Global NN} under the minimum requirement of $\d D\in C^2$, we will employ the homotopy formula methods  together with a Nash-Moser type iteration. These  techniques were originally employed by Webster \cites{MR999729, MR1128608, MR995504} to prove the classical Newlander-Nirenberg theorem, the CR vector bundle problem and the more difficult local CR embedding problem.  
These techniques together with a more precise interior regularity estimate for Henkin's integral solution operators of $\dbar_b$ on strictly pseudoconvex real hypersurfaces, have been successfully used by the second-named author and Webster \cites{MR2742034, MR2829316, MR2868966}  to obtain a sharp version of CR vector bundle problem and local CR embedding problem. The second-named author~\cite{MR3961327}  recently obtained a parameter version of Frobenius-Nirenberg theorem
by using similar techniques. We also mention the work of Polyakov \cite{MR2088929} who used similar techniques and obtained CR embeddings for a small perturbation of CR structures on compact regular 3-pseudoconcave CR submanifold $M$ of some complex manifold $X$ with $H^1 (M, T^{1,0}X|_M)=0$.

\medskip

The scheme of the proof of Theorem 1.1 is similar to the previous related work. However, we mention new features in the present work.  First is the use of 
the estimate of gaining $\frac{1}{2}$ derivative for homotopy operators on the closure of a $C^2$ strictly pseudoconvex domain proved recently by the second-named author~\cite{MR3961327}. Note that in previous mentioned work, interior regularity estimates of $\dbar, \dbar_b$ for homotopy formulas are used instead. Another important difference is that the Nash-Moser smoothing operator  \cite{MR0147741} was applied to the interior of the domains before.
In our case, we must find a way to use the Nash-Moser smoothing operator for
the closure of the domain $D$ since we are seeking global coordinate systems defined on $\overline D$.
To use the smoothing, we simply extend the original complex structure to a neighborhood of $\ov D$; this simple extension, however, does not preserve the formal integrability of the 
extended
 complex structure 
outside  $D$.
 The failure of the integrability is measured by the commutator $[\db, E]$ where $E$ is an extension operator for functions on $\ov D$ constructed by Stein~\cite{MR0290095}.  We shall make essential use of the vanishing order of $[\dbar, E]$ in our estimates. (See Sections \ref{sec  Settings} and \ref{four term estimate} for details.) We remark here that this commutator term is the main source for losing one derivative in our results.  The important commutator $[\dbar, E]$ was introduced by Peters \cite{MR1001710} and has been used by Michel \cite{MR1134587}, Michel-Shaw \cite{MR1671846} and others.  It is also one of the main ingredients in the $\frac{1}{2}$-gain estimate~\cite{MR3961327} for a homotopy operator on a
 strictly pseudoconvex domain with $C^2$ boundary.

The plan of the paper is as follows.  In Section \ref{Local_NN}, we first derive \rt{Local NN} from Theorem~\ref{Global NN}. In particular, we show that Catlin-Hanges-Jacobowitz's theorem is a  consequence of Hamilton's theorem,
the stability of  $\de_r(D)$ from Theorem \ref{Global NN},
 together with an initial normalization process constructed in Section \ref{Local_NN}.  In Section \ref{sec  Settings}, we recall basic facts about the standard H\"older-Zygmund norms, the Stein extension operator, Nash-Moser  smoothing operators and homotopy operators in~\cite{MR3961327}. In section \ref{approximation_of_embedding}, we derive an approximate solution of the embedding via the homotopy formula. We then obtain necessary estimates for the approximate solution and the  new almost complex structure in Sections~\ref{sec correction estimate} and  \ref{four term estimate}. In Section \ref{sec iteration}, we describe the iteration scheme and verify the induction hypotheseses. Finally, the convergence proof is carried out in Section \ref{convergence}.

\section{A reduction for Local Newlander-Nirenberg theorem  with boundary}\label{Local_NN}
\setcounter{thm}{0}\setcounter{equation}{0}

In this section, we derive Theorem~\ref{Local NN} by using \rt{Global NN}.  To achieve this, we need some preparations. First, we show that one can define the 
strict pseudoconvexity with respect to the standard complex structure instead of the given almost complex structure
near a reference point. 
Then we apply non-isotropic dilations to achieve the initial normalization condition :  $\znorm{X_{\overline\a}z}_{\cdot}\leq \de_r$  (the norm will be specified later), while the dilated hypersurface is close to the Heisenberg group near a reference point in $C^2$ norm.  Here $\de_r$ is the stability constant in Theorem \ref{Global NN} for some limiting domain under a non-isotropic dilation process. Finally, we construct a relatively compact $C^2$ strictly pseudoconvex domain $U$ which shares part of the boundary with $M$ and apply Theorem \ref{Global NN} to $(U, 
\{\widetilde X_{\ov\a}\}_{\a=1}^n)$
where $\{\widetilde X_{\ov\a}\}_{\a=1}^n$ is some suitable basis for the almost complex structure after dilation. 
We point out that the stability of $\de_r$ under $C^2$ perturbation is crucial for this argument to work.

Throughout the paper, the Greek letters $\a,\b,\g$ etc have range $1,2,\dots, n$ and roman indices $j,k$ etc have range $1,2,\dots, n-1$. We denote by $C_1,C_2,$ etc. constants bigger than $1$ and $c_1,c_2$, etc positive constants less than $1$.
We denote by $z=(z_1,\dots, z_n)$  the standard coordinates of $\C^n$, while the standard complex structure on $\C^n$ is defined by $\partial_{\overline \a}:=\frac{\partial}{\partial \overline z_\a},1\leq \a\leq n$. Set $\partial_{ \a}:=\frac{\partial}{\partial  z_ \a}$. 

We will use various constants $C(D), \delta(D)$, etc, which depend only on a domain $D$. 
Thus it will be convenient to apply some standard procedures to find defining functions of a domain. 
 A bounded domain $D$ in $\C^n$ with $C^k$ boundary with $k\geq1$ is defined by a $C^k$ function $\rho$ on $\C^n$. Thus $D$ is defined by $\rho<0$ and $\nabla\rho\neq0$ on $\pd D$.  
Locally $D$ is defined by a $C^k$ graph function $R$ via 
$$
\rho=-y_n+R(z',x_n)<0
$$
after permuting  $z_1,\dots, z_n$ and replace $z_n$ by $iz_n$ or $- iz_n$ if necessary. The collection of such functions will be denoted by $\{R_i\}$. Using a partition of unity, we can construct a defining function $\rho$ from the collection $\{R_i\}$. We may assume that $\rho=1$ away from a neighborhood of $\ov D$.  In such a way one can construct a defining function of $\rho$ which depends only on $D$ and we  shall call such a
 $\rho$ a {\em standard} defining function of $D$.

We start with the following elementary lemma showing how an almost complex structure changes with respect to a transformation of the form $F=I+f$ where $I$ is the identity mapping. This lemma is essentially in Webster \cite{MR999729}; however, we present it here for convenience of the reader and for our later proofs.

\begin{lemma}\label{elem_pp}
Let $\{X_{\overline\a}\}_{\a=1}^n$ be a $C^1$ almost complex structure defined near the origin of $\R^{2n}$.
\bppp
\item  By a $\R$-linear change of coordinates of $\C^n$, the almost complex structure $ \{X_{\overline\a}\}_{\a=1}^n$ can be transformed into $X_{\overline\a} = \d_{\overline\a} + A_{\overline\a}^\b \d_\b $ with $A(0) = 0$.
\item  Let $F=I+f$ be  a $C^1$ map 
with $f(0)=0$ and
 $Df$ small. The associated complex structure  $\{dF(X_{\overline\a})\}$ has a basis $\{X_{\overline\a}'\}$ such that  $X_{\overline\a}'=\d_{\overline\a}+{A'}_{\overline\a}^\b \d_{\b}$.
Moreover, $X_{\overline\a}F^\b=(X_{\overline\a}F^{\overline\g})({A'}_{\overline\g}^\b\circ F)$. Equivalently, in the matrix form,
\eq{AdfA}
A(z)+\d_{\overline z} f+A(z)\d_{z}f=(I+\d_{\overline z}\overline{ f(z)}+A(z) \d_z\overline{ f(z)})A'\circ F(z).
\eeq
\eppp
\end{lemma}
We remark the formula in $(ii)$
 is valid when $F=I+f$ is a diffeomorphism of $\overline D$ onto $\ov{D'}$  when $\norm{f}_{D,1}$ is sufficiently small. 
\begin{proof}
$(i)$
 Let $U_{\a} = \frac{1}{2}(X_{\a} + X_{\overline\a})$ and $V_{\a} = \frac{\sqrt{-1}}{2} (X_{\overline\a} - X_\a)$.
 We would like to find another coordinate system $w_\a = u_\a +\sqrt{-1} v_\a$ such that $\frac{\d}{\d u_\a}|_{w=0} = U_\a(0)$ and $\frac{\d}{\d v_\a}|_{w =0} = V_\a(0)$. Since $\{U_\a, V_\a\}$ and $\{\frac{\d}{\d x_\a}, \frac{\d}{\d y_\a}\}$ both span $T_0\C^n$,  then at $0$ we have
$$U_\a = a_\a^\b \frac{\d}{\d x_\b}+ b_\a^\b \frac{\d}{\d y_\b},\quad
 V_\a = c_\a^\b \frac{\d}{\d x_\b}+ d^\b_\a \frac{\d}{\d y_\b},
 $$
where the coefficient matrix  $\begin{bmatrix}a & b \\ c & d\end{bmatrix}$ is invertible. Set $x = au+bv$ and $ y= cu+dv$. In new variables $w$, we have $X_{\overline\a} = A_{\overline\a}^{\overline\g} \d_{\overline\g} + A_{\overline\a}^{\b} \d_{\b}$ where $(A_{\overline\a}^{\overline\g}(0))$ is the identity matrix $( \de_{\a}^\g)$ and $A_{\overline\a}^\b(0) = 0$. In particular, $(A_{\overline\a}^{\overline\g})$ is invertible near $0$. We can then use a linear combination  to achieve $X_{\overline\a} = \d_{\overline\a} + A_{\overline\a}^\b\d_\b$ with $A(0)=0$.

$(ii)$  Let us  show the existence of such a basis by determining the coefficient matrix $A'$. Since $T^{1,0}_{dF(X_{\overline\a})} = dF T^{1,0}_X$,  we know $F_*X_{\overline\a} = C_{\overline\a}^{\overline\b}  X'_{\overline\b}$ for some invertible matrix $(C_{\overline\a}^{\overline\b})$. Apply both sides to $F^{\overline\b}$  and use $X_{\overline\b}=\d_{\overline\b}+{A'}_{\overline\b}^\g \d_{\g}'$. Then we have $(C_{\overline\a}^{\overline\b})=(X_{\overline\a} F^{\overline\b})$, which is invertible when $Df$ is small. Consequently,
$$
X_{\overline\a} F^\b\d_\b
+X_{\overline\a} F^{\overline\b} \d_{\overline\b} =F_* X_{\overline\a}=
C_{\overline\a} ^{\overline\b}  X'_{\overline\b}.  $$
 Comparing the coefficients of $\d_{\overline\b}$ and $\d_\b$, we see that
$C_{\overline\a} ^{\overline\b}=X_{\overline\a}F^{\overline\b}$ and
   $$X_{\overline\a}F^\b=
 (X_{\overline\a} F^{\overline\g})(A_{\overline\g}'^\b\circ F).
  $$
  Since $(C_{\overline\a}^{\overline\g})$ is invertible, we have $A'^\b_{\overline\g} = (C^{-1})_{\overline\g}^{\overline\a} (X_{\overline\a}F^\b)\circ F^{-1} $.

Now identity \eqref{AdfA} follows from
 \gan
X_{\overline\a} F^\b= (\d_{\overline\a}+A_{\overline\a}^\g \d_\g)(z^\b+f^\b)= A_{\overline\a}^\b +\d_{\overline\a}f^\b+A_{\overline\a}^\g\d_\g f^\b= (A+\dbar f +A\d f)_\a^{\overline\b},\\
(X_{\overline\a}F^{\overline\g})(A_{\overline\g}'^\b\circ F)=(\d_{\overline\a}+A_{\overline\a}^\g\d_\g )(z^{\overline\b}+f^{\overline\b})(A'\circ F)=((I+\dbar \overline f + A\d \overline f)(A'\circ F))_\a^{\overline\b}.
\qedhere
\end{gather*}
\end{proof}

Using the integrability condition, we now remove the first order term in the Taylor expansion of $A$ at the origin.

\begin{lemma}\label{A=O(2)}
Suppose that a $C^1$ almost complex structure defined by $\{X_{\overline\a}\}$ with $X_{\ov\a}= \d_{\overline\a}+ A_{\overline\a}^\b \d_\b$   and  $A(0)=0$ satisfies the integrability condition \eqref{integrability} at $0$. Then we can make a polynomial change of coordinates such that in the new coordinate system, the almost complex structure is given by $X'_{\overline\a}= \d_{\overline\a}+ A'^{\b} _{\overline\a}\d_\b$ with $A'(0)=0$ and $DA'(0)=0$.
\end{lemma}
\begin{proof} Throughout the paper, we write $f(\zeta)=o(|\zeta|^k)$ if $\lim_{|\zeta|\to 0}f(x)/{|\zeta|^k}=0$, where $\zeta$ are real or complex variables.
We make a polynomial change of coordinates  $F=I+f$ where 
$$
f^\b = -\d_\g A^\b_{\overline\a }(0) z^\g \overline z^\a-\frac{1}{2}\d_{\overline \g}A_{\overline\a}^\b(0) \overline z^\g \overline z^\a.
$$
According to Lemma~\ref{elem_pp}, we have\[
A+\dbar f+A\d f=(I+\overline{\d f}+A\d \overline f)A'\circ F.
\]
Shrinking the domain if necessary, we can assume that $(I+\overline{\d f}+A\d \overline f)$ and $F$ are invertible. Therefore, in order to show $A'(z)=o(|z|)$, it suffices to show that $A(z)+\dbar f(z)+A(z)\d \overline f(z) = o(|z|)$.

Since our structure satisfies the integrability condition at $0$, then $[X_{\overline\a}, X_{\overline\b}]$ is in the span of $X_{\overline 1},\cdots, X_{\overline n}$ at $0$. This implies that via $A(0)=0$,
\eq{integrability@0}
\d_{\overline\a}A_{\overline\b}^\g(0)
=\d_{\overline\b}A_{\overline\a}^\g(0).
\eeq
Plugging \eqref{integrability@0} into $A+\dbar f+ A\d f$, we get 
\al
A_{\overline\a}^\b&+\d_{\overline\a}f^\b+A_{\overline\a}^\g \d_\g f^\b \\ \nonumber
&= A_{\overline\a}^\b - \d_\g A_{\overline\a}^\b (0) z^\g -\frac{1}{2}\d_{\overline \g }A_{\overline\a}^\b(0) \overline z ^\g -\frac{1}{2}\d_{\overline\a} A_{\overline \g }^\b(0) \overline z^\g- A_{\overline\a}^\rho \d_\rho A_{\overline \g}^\b(0) \overline z^\g\\
&=A_{\overline\a}^\b - \d_\g A_{\overline \a}^\b (0) z^\g -\d_{\overline \g}A_{\overline\a}^\b(0) \overline z^\g-  A_{\overline\a}^\rho \d_\rho A_{\overline \g}^\b(0) \overline z^\g,
\nonumber
\end{align}
where we have used the integrability condition at $0$ in the second equality. Now it is clear that the right-hand side vanishes up to second order at the origin and thus the lemma follows.
\end{proof}

Note that in the formulation of Theorem~\ref{Local NN}, we require the boundary to be strictly pseudoconvex with respect to the given almost complex structure. However, in order to apply Theorem \ref{Global NN}, it is important that the boundary is strictly pseudoconvex with respect to the standard complex structure. Next lemma shows that these two conditions are locally equivalent provided the given structure and the standard one agree up to second order at a reference point, after some initial normalization.

\begin{lemma}\label{Reformulation}
Let $M\subset\d U$ be a $C^2$ real hypersurface.  Let $X_{\overline\a}= \d_{\overline\a}+A_{\overline\a}^\b \d_\b$, $\a=1,\cdots, n$ define a formally
integrable
$C^1$   complex structure on the one-sided domain $U\cup M$. Suppose that $0\in M$ and $A(z)=o(|z|)$. Assume that $M$ is strictly pseudoconvex with respect to  $(U\cup M, \{X_{\overline\a}\})$  $($see \eqref{spc_X} for definition$)$. The following hold.
\bppp
\item  $M$ is strictly pseudoconvex  with respect to the standard complex structure near the origin.
\item After a local polynomial change of coordinates that preserves the condition $A(z) = o(|z|)$, there exists a defining function $r$ for $M$, defined near the origin, such that $\rho<0$ on $U$, $\rho=0$ on $M$, 
 and 
  \[\rho(z)=-y_n+|z'|^2+h(z',x_n)\]
where 
$h = o(2)$ is a $C^2$ function.
\eppp
\end{lemma}
\begin{proof} 

$(i)$   Since $X_{\overline\a}= \d_{\overline\a}+A_{\overline\a}^\b \d_\b$, $\a=1,2\cdots, n$ form a basis of $T^{0,1}_X \C^n$ near $0$, we can find the dual frames  $\om^\a, \om^{\overline\a}$ of $X_{\a}, X_{\overline\a}$ near $0$. Let $\Lambda_{0,X}^{p,q}$ denote the germ of smooth $(p,q)$ forms with respect to the almost complex structure $X$ at $0$. More precisely, $u = u_{I\overline J} \om^I \we \om^{\overline J} \in \Lambda^{p,q}_{0,X}$ where $I,J$ are multi-indices, $|I|= p, |J| = q$ and $u_{I\overline J}$ are elements in the germ of smooth functions at $0$. Denote the decomposition of the exterior derivative with respect to $\{X_{\overline\a}\}$ by $d=\d_X+\dbar_X$. See \cite{kobayashi1996foundations}*{P. 126}. Thus for a function $r$,
\ga
dr=X_\a r\om^\a+X_{\ov\a}r\om^{\ov\a},\\
\partial_X\ov \partial_X r=(X_\a X_{\ov\b}r + X_{\ov\g}r C^{\ov\g}_{\a\ov\b})\om^\a\wedge\om^{\ov\b},
\end{gather}
where $C_{a\overline \b}^{\overline \g} = -\om^{\overline\g}([X_{\a}, X_{\overline\b}])$. Notice in particular that, $C_{a\overline \b}^{\overline \g}(0) = 0$ since $A = o(|z|)$.

According to \eqref{spc_X}, we need to show  that 
$\sqrt{-1}\d_X\dbar_X r(0)$ is positive-definite on $ \mathbb{L}_X$ if and only if  $\sqrt{-1}\d\dbar \rho(0)>0$ on  ${\mathbb{L}}$ where $\mathbb{L}_X = T^{1,0}_{X} \C^n \cap \C T_0M$ and   $\mathbb{L} = T^{1,0}\C^n \cap \C T_0M$.

Write $r(z)=\IM(\sqrt{-1} c_\a z^\a)+O(2)$. Since $d \rho \ne 0$ at $0$, we may assume $\frac{\d \rho}{\d y_n}\neq 0$  by permuting coordinates which preserves conditions on $A_{\ov\a}^\b$. Consequently, the linear transformation  $(z',z^n)\to (z', \sqrt{-1}  c_\a z^\a)$ preserves $A (z)= o(|z|)$ and we have\[
\rho=
-y_n+O(2).
\]
Applying the implicit function theorem to $\rho(z',x_n,y_n)=0$, we obtain $y_n=F(z',x_n)$ where $F=O(2)$. More precisely, we can write\[
\begin{aligned}
y_n= F(z',x_n)&=
 a_{j\overline k} z^j \overline z^j+2\RE(
 b_{jk}z^j z^k+ 
 c_{j }z^j x^n)+ o(|z'|^2+x_n^2)\\
&=
a_{j\overline k} z^j \overline z^k+2\RE(
 b_{jk}z^j z^k+
 c_{j}z^j z^n)+ o(|z'|^2 + x_n^2)\\
\end{aligned}
\]
where in the second line, we use the fact that $F = O(2)$.

Then it is easy to see that $L_i = X_i - \frac{X_i \rho}{X_n \rho}X_n$, $i=1,\cdots, n-1$ form a basis for $\mathbb{L}_X$ near $0$. Therefore, near the origin, we have
$$ 
\d_X \dbar_X \rho (L_i, L_{\overline j})=\sum_{\a,\b,\g} (X_\a X_{\overline\b} \rho + X_{\ov\g}\rho C_{a\overline \b}^{\overline \g}) (\om^\a\we \om^{\overline \b}) (X_i - \frac{X_i \rho}{X_n \rho}X_n,  X_{\overline j} - \frac{X_{\overline j} \rho}{X_{\overline n} \rho}X_{\overline n}).
$$ 
Using the fact that $A=o(|z|)$, $C_{\a\overline\b}^{\overline\g}(0)=0$ and $\frac{\d \rho}{\d z_i}(0)=0$ for $i = 1,2,\cdots, n-1$, we have
\eq{}
\d_X \dbar_X \rho (L_i, L_{\overline j})(0)= (\d_\a\d_{\overline\b}\rho)(0) (dz^\a \we dz^{\overline\b})(\d_i, \d_{\overline j}).
\eeq
The proof is then complete by noticing that $\{\d_i\}_{i=1}^{n-1}$ form a basis of $\mathbb{L}$ at $0$ with respect to the standard complex structure.

$(ii)$ According to the first part, it suffices to show the conclusion for a strictly pseudoconvex $C^2$ real hypersurface in $\C^n$. The proof is standard, but we include it here to ensure that the condition $A = o(|z|)$ is preserved which is required in the next lemma.
 
Let us make  a second order change of coordinates
$$z'\to z', \quad z^n\to z^n-\sqrt{-1}\left( b_{jk}z^j z^k+ c_{\a n}z^\a z^n\right),
$$
which clearly preserves the condition $A (z)= o(|z|)$ according to part $(ii)$ of Lemma \ref{elem_pp}. We have
\[
\rho=-y_n+
 a_{jk}z^j \overline z^k+h(z', x_n), \quad h(z',x_n) = o(|z'|^2 + |x_n|^2).
\]

Since $M$ is strictly pseudoconvex, we see that the hermitian matrix $(a_{j\overline k})$ is  positive definite in $z'$. The final expression then follows from a complex linear change of coordinates. Note that the latter also preserves $A(z) = o(|z|)$.
It is clear from our construction that we still have $h(0) = Dh(0) = D^2 h(0) =0$.
\end{proof}

We can now reformulate Theorem~\ref{Local NN} in an equivalent form that requires the boundary to be strictly pseudoconvex with respect to the standard complex structure. Indeed, since the integrability condition of our almost complex structure holds at the origin by continuity, we can assume that $A=o(|z|)$ by \rl{A=O(2)}. Then according to Lemma \ref{Reformulation}, the two assumptions are equivalent.
Next, we achieve initial normalization by a non-isotropic dilation.

\begin{prop}\label{InitialNorm}
Let $M\subset \d U$ be a $C^2$ real hypersurface containing the origin.  Let $X_{\overline\a}= \d_{\overline\a}+A_{\overline\a}^\b \d_\b$, $\a=1,\cdots, n$ be $\Lambda^r(U\cup M)$, $1< r <\infty$, integrable almost complex structure defined on the one-sided domain $U\cup M$ with $A(z)=o(|z|)$.  Assume that $M$ is strictly pseudoconvex with respect to $(U\cup M, \{X_{\overline\a}\})$. Then after a non-isotropic dilation $\var_\e (z', z_n)= (\e^{-1} z', \e^{-2} z_n)$ where $\e>0$ is sufficiently small, we have the following
\bppp
\item  there exists some open set $B \subset \C^n$ and a $C^2$ function 
$\rho_\e :   B\to \R$, such that $D_\e = \{z\in B :   \rho_\e(z) < 0\}  \subset \var_\e(U\cup M)$ is a connected $C^2$ strictly pseudoconvex domain which shares part of the boundary with $\var_\e(M)$ near origin. Moreover, there exists a $C^2$ function $\rho_0 :  B\to \R$  such that $\lim_{\e\to 0}\norm{\rho_\e - \rho_0}_{B,2} = 0$ and $D_0 \defeq \{z\in B :  \rho_0(z) < 0\}$ is also a connected  $C^2$ strictly pseudoconvex domain,
\item  on $\ov D_\e$ each $d\var_\e X_{\ov\b}$ is spanned by  $X^\e_{\ov\a}=\pd_{\ov\a}+(A^{(\e)})^{\b}_{\ov\a}\pd_{\b}$, $\a=1,\dots, n$, where $\znorm{ A^{(\e)}}_{D_\e,r'}$ tends to $0$ with $\e$ for any finite $r'\leq r$.
\eppp
\end{prop}
\begin{proof}
$(i)$  By the  second part of Lemma \ref{Reformulation},  we can assume that  
the defining function of $M\cup U$ is locally given by 
\[
\rho(z) = -y_n +|z'|^2 + h(z', x_n), \quad \text{$h = o(|z'|^2 + x_n^2)$},
\]
while the condition $A(z)= o(|z|)$ is preserved. Let us denote
$B_{a} = \{(z',z_n)\in \C^n  : |z|< a\}$
 for any $a>0$.

We apply the dilation $\var_\e$ to $U\cup M$. Then the new defining function for $M_\e\defeq \var_\e(M)$  can locally be written as
 \[
\hat \rho_\e (z) = - y_n + |z'|^2 + \e^{-2} h(\e z', \e^{2}x_n),\quad h=o(|z'|^2+x_n^2).
\]

Without loss of generality, we may assume that $\hat \rho_\e$ are defined on $B_2$.
Moreover, we have $\var_\e(U\cup M)\cap B_2=\{z\in B_2:  \hat \rho_\e(z) < 0\}$.
Then we shall construct a $C^2$ strictly pseudoconvex domain $D_\e$ such that
$$
B_1\cap \var_\e(U\cup M)\subset  \ov{D_\e} \subset B_2 \cap \ov{\var_\e(U\cup M)}$$
as follows.

  Let $\chi:  \R\to \R^{+}$ be a smooth non-decreasing convex function such that $\chi = 0 $ on $(-\infty , 1]$ and
  $\chi(4)=1$.
  Moreover,  we assume $0< \chi'(x)\leq 1$ for $x\in (1,4)$.  Define $D_\e = \{z\in B_2:  \rho_\e(z)<0\}$ where
 \[
 \rho_\e (z)= -y_n + |z'|^2 +  \e^{-2} h(\e z', \e^{2}x_n) + 5\chi (|z|^2).
\]
Note that $D_\e \subset \var_\e(U\cup M)$ since we have added a non-negative term to $\hat \rho_\e$.

We check that $D_\e$ satisfies all the requirements. Since $\chi = 0$ on $(-\infty, 1]$,
 one has $B_1\cap \var_\e(U\cup M) \subset D_\e$. In order to show $\ov {D_\e} \subset B_2\cap \ov{\var_\e(U\cup M)}$, it suffices to show that for all $z\in \d B_2$, we have $r_\e(z) > 0$, that is that
  \[
 5\chi (4) = 5 >  y_n - |z'|^2 -  \e^{-2} h(\e z', \e^{2}x_n).
\]
This holds for $\e$ small since $h(z',x_n)= o(|z'|^2 + x_n^2)$.

 We want to prove  that $\rho_\e$ defines a strictly pseudoconvex domain with $C^2$ boundary.

Clearly, $\rho_\e$ is a $C^2$ function. It then suffices to show that
$(a).$~$d \rho_\e (z) \ne 0$ for all $z\in \d D_\e$;
$(b).$~$\lambda:=\inf\frac{\d^2 \rho_\e}{\d z_i\d\overline z_j}(z) t_i t_{\overline j}>0$ where the infimum is taken for $z\in \d D_\e$,  $\sum_{j=1}^n t_j \d_j \rho_\e =0$ and $|t|=1$.

 Note that $-y_n+|z'|^2$ and $\chi(|z|^2)$ are plurisubharmonic. On $\pd D_\e\cap M_\e$, we have $\rho_\e=\hat \rho_\e$.  It is clear that there is a positive constant $c_0$ such that when $\e$ is small,
 $$
 \lambda_0:=\inf_{t\in T_xM_\e, |t|=1,x\in\pd D_\e\cap M_\e}\frac{\d^2 \hat \rho_\e}{\d z_i\d\overline z_j}(z) t_i t_{\overline j}>c_0.
 $$
 We can find a neighborhood $N$ of $M_\e\cap \pd D_\e$, independent of $\e$,
 such that
$$
 \lambda_1:=\inf_{t\in T_xM, |t|=1,x\in\pd D^\e\cap N}\frac{\d^2 \rho_\e}{\d z_i\d\overline z_j}(z) t_i t_{\overline j}>c_0/2.
 $$
On $D_\e\setminus N$, we have $-y_n+|z'|^2\leq -c_0'$, where $c_0$ is a positive constant, and hence
$$
\chi(|z|^2)>c_0'/2, \quad z\in \pd D_\e\setminus N.
$$
Note that $\chi(|z|^2)$ is strictly plurisubharmonic at $z$ when $\chi(|z|^2)>0$. Therefore,
$$
 \lambda_2:=\inf_{t\in T_xM_\e, |t|=1,x\in\pd D_\e\setminus N}\frac{\d^2 \rho_\e}{\d z_i\d\overline z_j}(z) t_i t_{\overline j}>c_1>0.
 $$
This shows that $\lambda>\min(\lambda_0,\lambda_1,\lambda_2)/2$ when $\e$ is small. Therefore, $D_\e$ is a $C^2$ strictly pseudoconvex domain.

It is obvious that $\lim_{\e\to 0}\norm{\rho_\e - \rho_0}_{B,2} = 0$.
Note that $r_0$ is a convex function. Thus $D_0$ is connected and the connectedness of $D_\e$ follows easily from the $C^2$ convergence.

$(ii)$ To find the new vector fields, we let $X^{(\e)}_{\overline\a}= \e(\var_\e)_* (X_{\overline\a})$, $ X^{(\e)}_{\overline n}= \e^{2}(\var_\e)_*(X_{\overline n})$. Then for $1\leq j,k \leq n-1$, we have
\ga
X^{(\e)}_{\overline j}= \d_{\overline j}'+ (A^{(\e)})_{\overline j}^k \d_{k}'+\e^{-1} (A^{(\e)})_{\overline j}^n \d_n',\\
X^{(\e)}_{\overline n}=  \d'_{\overline n}+ \e (A^{(\e)})_{\overline n}^k \d'_k + (A^{(\e)})_{\overline n}^n\d'_n
\end{gather}
where $\d_{\cdot}$ are vector fields associated to new coordinate and $(A^{(\e)})(z) \defeq A (\e z', \e^{2}z_n)$.

Since $A(z) = o(|z|)$, then $\znorm{ A^{(\e)}}_{D_\e,r'}\to0$ as $\epsilon\to0$ for any finite $r' \leq r$.
\end{proof}

\

Assuming that Theorem \ref{Global NN} holds, we are now ready to prove Theorem~\ref{Local NN}.

\medskip

{\it{Proof of \rta{Local NN}.}} \
According to Lemma \ref{Reformulation}, we may assume without loss of generality that 
$M\cup U = \{z\in U_0 : \rho(z)\leq 0\}$ where $\rho(z)= -y_n+ |z'|^2+h(z',x_n)$, $h=o(2)$  and $h\in C^2(V)$ on some neighborhood $V$ of the origin.

First, we let $5<r<\infty$ and apply Lemma \ref{InitialNorm} to $\{
U\cup M, \{X_{\overline\a}\}_{\a=1}^n\}$ with $\e$ to be determined. Then we obtain a $C^2$ strictly pseudoconvex domain $ D_\e \subset \var_\e(U\cup M)$ which shares part of the boundary with $M_\e$. Moreover, there exists a new basis $ \{X^{(\e)}_{\overline \a}\}_{\a=1}^n\in C^r(\ov{D_\e})$ for $\{ (\var_{\e})_*X^{(\e)}_{\overline 1}, \cdots,  (\var_{\e})_*X^{(\e)}_{\overline n}\}$ such that
$\znorm{ A^{(\e)}}_{D_\e,r}$ tends to $0$ as $\e\to 0$ and a limiting $C^2$ strictly pseudoconvex domain with defining function $\rho_0$ such that $\norm{\rho_\e - \rho_0}_{B_2,2}$ tends to $0$ as $\e\to 0$.

Fix $2<s<3$. According to Theorem \ref{Global NN}, there exists $\de_r(D_0)>0$ that is lower
stable under a small $C^2$ perturbation of $D_0$.
Therefore, we can find $\e$ sufficiently small such that 
\ga
 \znorm{ A^{(\e)}}_{D_\e,r}\leq 
 \de_r(D_0) / C(D_0), \\
\de_r(D_0) \leq  C
(D_0)\de_r(r_\e).
\label{firsteq}
\end{gather}
Here $\e$ is chosen for the $C(D_0)$ in \eqref{firsteq}.  Therefore, we have
$$
 \znorm{ A^{(\e)}}_{D_\e,r}\leq \de_r(r_\e).
$$

Consequently, we are able to apply Theorem \ref{Global NN} to $(D_\e,X^{(\e)}_{\overline\a})$ to obtain a $\Lambda^{r-1}(\ov {D_\e})$ diffeomorphism $F_\e :  D_\e\to \C^n$ onto its image that sends the almost complex structure to the standard one. Since $D_\e$ shares part of the boundary with $M_\e$, $F_\e$ induces a diffeomorphism near $0\in M_\e$ which sends the integrable almost complex structure to the standard one on one side of the domain.  The embedding $F$ is then given by $F_\e \circ \var_\e^{-1}$.

Finally, we consider the case $r=\infty$. Notice that merely  $\znorm{A}_{D,5}\leq \de_5(D_0)$ is required for the statement of Theorem \ref{Global NN} to be valid. Therefore, we do not need to control higher order derivatives of the error and the previous argument still applies. The proof of the Theorem \ref{Local NN} is complete.

\section{Preliminaries}\label{sec  Settings}
\setcounter{thm}{0}\setcounter{equation}{0}

In this section, we present some preliminaries for the proof of Theorem \ref{Global NN}. First, we recall some basic results for standard H\"older norms $\norm{\cdot}_{D,a}$, $0\leq a <\infty$  and H\"older-Zygmund norms $\znorm{\cdot}_{D,a}$, $0< a <\infty$ on domains $D\subset \R^n$ with cone property.  We then introduce three main tools used in the proof : the Stein extension operator, the Nash-Moser smoothing operator and the homotopy formula on strictly pseudoconvex domain with $C^2$ boundary in~\cite{MR3961327}. We also include necessary estimates for these operators for later use.

\subsection{Convexity of H\"older-Zygumund norms}
We say that a domain $D$ in $\R^n$ has the cone property if there exists some $C_* = C_*(D)>0$ such that the following hold.
\begin{enumerate}
\item Given two points $p_0, p_1$ in $D$, there exists a piecewise $C^1$ curve $\g(t)$ in $D$ such that $\g(0)= p_0$ and $\g(1)= p_1$, $|\g'(t)|\leq C_* |p_1-p_0|$ for all $t$ except finitely many values.
\item  For each point $x\in \ov{D}$, $D$ contains a cone $V$ with vertex $x$, opening $\th> C_*^{-1}$ and height $h> C_*^{-1}$.
\item  The diameter of $D$ is less than $C_*$.
\end{enumerate}

For a domain with cone property, the following H\"older estimates for interpolation, product rule and chain rule  are well-known. For instance, see the 
appendices of \cites{MR2829316, MR3961327} or \cite{MR0602181} for proofs and more details:
\ga
\norm{u}_{D,
 (1-\th) a + \th b}\leq C_{a,b}\norm{u}_{D, a}^{1-\th}\norm{u}^{\th}_{D,b},\quad\text{ $0\leq \th\leq 1$}, \label{convexity}\\
\norm{uv}_{D,a}\leq C_a (\norm{u}_{D,a}\norm{v}_{D,0}+\norm{u}_{D,0}\norm{v}_{D,a}),\label{product rule}\\
\norm{u\circ g}_{D,a}\leq C_a (\norm{u}_{\widetilde D,a}\norm{g}_{D,1}^a+\norm{u}_{\widetilde D,1}\norm{g}_{D,a}+\norm{u}_{\widetilde D,0}).\label{chain rule}
\end{gather}
If $(a,b)=\th(a_1,b_1)+(1-\th)(a_2,b_2), 0\leq \th\leq 1$, we have
\eq{holder_only}
\norm{u}_{D,a}\norm{v}_{D',b}\leq C_{a,b}(\norm{u}_{D,a_1}\norm{v}_{D',b_1}+\norm{u}_{D,a_2}\norm{v}_{D',b_2}).
\eeq

We now recall the definition of H\"older-Zygmund spaces  and basic properties.
 For $0< r\leq 1$,  the  H\"older-Zygmund space $\Lambda^r(\rr^n)$ is the set of functions $f\in L^\infty(\rr^n)$ such that
\eq{frrn}
|f|_{\rr^n,r}:=
|f|_{L^\infty(\rr^n)}+\sup_{y\in\rr^n}\f{|\Delta^2_yf|_{L^\infty(\rr^n)}}{|y|^r}.
\eeq
Here $\Del_yf(x):=f(x+y)-f(x)$ and thus $\Del^2_yf(x)=f(x+2y)+f(x)-2f(x+y)$.
When $r>1$, we define $\Lambda^r(\rr^n)$ to be the set of functions $f\in C^{[r]-1}(\rr^n)$ satisfying
$$
|f|_{\rr^n,r}:=|f|_{L^\infty(\rr^n)}+|\pd f|_{\rr^n,r-1}<\infty.
$$
 For a non-integer $r$,   $|\cdot|_{\rr^n,r}$ is equivalent to the H\"older norm $\|\cdot\|_{\rr^n;r}$;
 when $1<r<2$, $\znorm{\cdot}_{\rr^n,r}$ is also equivalent to the norm defined by \eqref{frrn}.
 See~\ci{MR0290095}*{Prop.~8, p.~146} and by the equivalence of the two norms one means
 $$
 c_r\|f\|_{\rr^n,r}\leq |f|_{\rr^n,r}\leq C_r \|f\|_{\rr^n,r}
 $$
 for two positive numbers $c_r,C_r$ depending only on $r$.
Clearly, we  have $c_r$ tends to $0$ when $r$ tends to positive integer and $C_r\leq2$.
 Let $F$ be a closed subset in $\rr^n$.  Let $r\in(0,\infty)$.
We write $f\in\Lambda^r(F)$ if there exists   $\widetilde f\in\Lambda^r(\rr^n)$ such that $\widetilde f|_F=f$. Define $|f|_{F,r}$  to be the infimum of  $|\widetilde f|_{\rr^n,r}$ for all such extensions $\widetilde f$.

Next, we recall the extension operator constructed by Stein \cite{MR0290095 }. Given a bounded domain $D\subset \R^n$ with Lipschitz boundary (i.e. the boundary is locally the graph of some Lipschitz function), there exists an extension operator
\eq{extension}
E   : \Lambda^{r}(\ov D)\to \Lambda^r(\R^n) \quad \text{with $\znorm{Ef}_{\rr^n,r}\leq C_r(D)\znorm{f}_{ \ov D,r}$}, \quad \forall r\in (0,\infty),
\eeq
where 
 the operator norm $C_r(D)$  depends only on the Lipschitz constant of the boundary.  In fact, Stein~\cite{MR0290095 } proved the above estimates for Sobolev spaces. The estimates $(\ref{extension})$ for Zygmund spaces can be found in~\cite{MR3961327}. We refer the reader to  \cites{MR3961327,MR0290095} for more details on the construction and the estimates.
When there is no confusion, we write $\znorm{\ \cdot\ }_{\ov D,r}$ as $\znorm{\ \cdot\ }_r$.

We now derive convexity in H\"older-Zygmund norms.
It is clear that the second term in \eqref{frrn} satisfies
$$
\sup_{y\in\rr^n}\f{|\Delta^2_yf|_{L^\infty(\rr^n)}}{|y|^{(1-\theta )r_0+\theta r_1}}\leq
\left(\sup_{y\in\rr^n}\f{|\Delta^2_yf|_{L^\infty(\rr^n)}}{|y|^{r_0}}\right)^{1-\theta}
\left(\sup_{y\in\rr^n}\f{|\Delta^2_yf|_{L^\infty(\rr^n)}}{|y|^{r_1}}\right)^{\theta}
$$
for $r_0,r_1\in(0,2)$ and $\theta\in(0,1)$. Then we get
\ga
\znorm{u}_{(1-\th)a +\th b}\leq C_{b-a}\znorm{u}_{ a}^{1-\th}\znorm{u}^{\th}_{b},\quad\text{ $0\leq \th\leq 1$}, \label{zconvexity-}
\end{gather}
for $0<b-a<2$ and $\theta\in(0,1)$. Hence, it also holds for all positive $a,b$; indeed, suppose it holds for $0<b-a<d$. Suppose $a<c<b$ with $d/2<b-a<3d/2$ and $c=(1-\theta )a+\theta b$. We take $e$ so that $0<\max (b-e,e-a)<d$. We may assume that $c<e$. Thus
$$
|u|_{c}\leq C|u|_{a}^{\frac{e-c}{e-a}}|u|_{e} ^{\frac{c-a}{e-a}}, \quad |u|_{e}\leq C|u|_{c}^{\frac{b-e}{b-c}}|u|_{b} ^{\frac{e-c}{b-c}}.
$$
Eliminating $|u|_e$ and solving for $|u|_c$, we get  \eqref{zconvexity-}.   By \eqref{extension} and \eqref{zconvexity-}, for $D$ we get
\ga
\znorm{u}_{D, (1-\th)a +\th b}\leq C_{a,b}(D)\znorm{u}_{D, a}^{1-\th}\znorm{u}^{\th}_{D,b},\qquad 0\leq \th\leq 1, \quad a,b>0. \label{zconvexity}
\end{gather}
To derive the H\"older-Zygmund version of the product and chain rules, we use an equivalent norm.

\pr{gs141} Let $0<r<\infty$.   Then $f\in\Lambda^r(\rr^n)$ if and only if there is a decomposition
$
f=\sum_{k\geq0} f_k
$
so that $f_k\in C^\infty(\rr^n)$ and
$$ 
 |\pd^i f_k|_{L^\infty(\rr^n)}\leq A2^{k(i-r)}, \quad i=0,\dots, [r]+1.
$$ 
Furthermore,  the smallest constant $A_r (f)$ of $A$ is equivalent with
$|f|_{\Lambda^r}$,
 i.e. $c_rA_r (f)\leq|f|_{\Lambda^r}\leq C_rA_r (f)$ for some positive numbers $c_r ,C_r$ independent of $f$.
\epr

For the proof, see \ci{MR0290095}. We now derive the following.

\le{zprod}
Let $D,\widetilde D$ be connected bounded domains with Lipschitz boundary and let $g$ maps $D$ into $\widetilde D$. Suppose that $\norm{g}_{D;1}<C$.  Then  we have
\al
\znorm{uv}_{D,a}&\leq C_a (\znorm{u}_{D,a}\norm{v}_{D,0}+\norm{u}_{D,0}\znorm{v}_{D,a}),\quad a>0;
\label{zproduct rule}
\\
|u\circ g|_{D,1}&\leq C(D)C(\widetilde D)|u|_{\widetilde D,1}(1+C_{1/\e}\norm{g}_{D,1+\e}^{\frac{1}{1+\e}});
\label{chain1}
\\
\label{zchain rule} 
\znorm{u\circ g}_{D,a}&\leq C_a (D)C_a(\widetilde D)(C_{1/\e}\znorm{u}_{\widetilde D,a}\norm{g}^{\frac{1}{1+\e}}_{D, 1+\e}\\
&
\qquad +\norm{u}_{\widetilde D,1}\znorm{g}_{D,a}+\norm{u}_{\widetilde D,0}), \quad a>1. 
\end{align}
Here $C_{1/\e}$ is a positive constant depending on $\e$ that tends to $\infty$ as $\e\to 0$.
\ele
\begin{proof} Note that  stronger inequalities for H\"older norms are given by \eqref{convexity}-\eqref{chain rule}.
We only need to verify the lemma when $a$ is an integer. Here we need a bit more 
for low order derivatives for $g$. We will also employ the Stein extension operator.

For the product rule, by Stein extension it suffices to consider the case $D=\rr^n$ while $u,v$ have compact support in a ball of fixed radius. By \rp{gs141}, we have $u=\sum u_k$, $v=\sum v_k$ with
$$
|\pd^\ell u_k(x)|\leq C_\ell|u|_a2^{-k(a-\ell )}, \quad |\pd^\ell v_k(x)|\leq C_\ell |v|_a2^{-k(a-\ell)}.
$$
We recall from the proof of \rp{gs141} the formula
$$
u_k(x)=\int\varphi_k(x-y)u(y)\, dy.
$$
Here $\var_k(x)=2^{mk}\var(2^kx)$
and $\var$ has compact support  in $\rr^m$ for $m=2n$. Thus
$$
|\pd_x^\ell u_k(x)|\leq C_\ell 2^{k\ell}\norm{u}_0.
$$
Decompose $uv=\sum w_k$ with $w_k=\sum_{i=0}^ku_iv_{k-i}$.
We bound $\sum_{i=0}^k |\pd_x^ju_i\pd_x^{\ell-j}v_{k-i}|$ above by
$$
 \sum_{i=0}^k2^{ij}\|u\|_0\znorm{v}_a2^{-(k-i)(a-(\ell-j))}= \|u\|_0|v|_a \sum_{i=0}^k2^{-k(a-\ell+j)-i(a-\ell)}.
$$
For $\ell<a$, the sum is bounded by $2^{-k(a-\ell+j)}\leq 2^{-k(a-\ell)}$.
For $\ell\geq a$ and $j>0$, the sum is bounded by
$$k\times 2^{-k(a-\ell+j)+k(\ell-a)}\leq k\times 2^{-kj}\leq C\leq C2^{-k(a-\ell)}.
$$
This gives us the desired estimate for $j>0$. By symmetry, for $j<\ell$ we get $$\sum_{i=0}^k|\pd_x^jv_{k-i}\pd_x^{\ell-j} u_i| \leq C_\ell 2^{-k(a-\ell)}\znorm{u}_a\norm{v}_0.
$$
We have verified \eqref{zproduct rule}.

We now verify \eqref{chain1}. Let $\widetilde u=E_{\widetilde D}u$ and let $\widetilde g=E_Dg$. Then
$$
|u|_{D,1}\leq |\widetilde u|_1\leq C(\widetilde D)|u|_{\widetilde D,1}, \quad |g|_{D;1+\e}\leq 2\norm{\widetilde g}_{1+\e}\leq
 C_\e(\widetilde D)|g|_{D,1+\e}.
$$
Thus $\widetilde u\circ\widetilde g$ is an extension of $u\circ g$. Let us drop all tildes in $\widetilde u,\widetilde g$.
We have
\aln
| u\circ  g(x+h)&+u\circ g(x-h)-2u\circ g(x)|\\
&\leq|u|_{1}|g(x+h)-g(x)|+
|u(g(x-h))-u(2g(x)-g(x+h))|\\
&\leq C_1\znorm{u}_{1}\norm{g}_{1}|h|+\norm{u}_{\frac{1}{1+\e}}(C_n\norm{g}_{1+\e}|h|^{1+\e})^{\frac{1}{1+\e}}
\\
&\leq C_1\znorm{u}_{1}\norm{g}_{1}|h|+C_nC_{1/\e}\znorm{u}_{1}
\norm{g}_{1+\e}^{\frac{1}{1+\e}}|h|.
\end{align*}
Here we have used
\ga
 g(x+h)+  g(x-h)-2  g(x)=h\cdot \int_0^1(\nabla   g(x+th)-\nabla   g(x-th)\, dt,
\\
\|g\|_{\a}\leq C_{1/\a}|g|_1, \quad 0<\a<1.
\end{gather}
Note that $C_{1/\a}$ is not bounded as $\a$ tends to $1^-$.
We have verified \eqref{chain1}.
  To verify  \eqref{zchain rule},  it remains to verify it for integer $a\geq2$. We  have
  \eq{pda1}
  \pd^{a-1}(u\circ g)=(\pd^{a-1}u)\circ g\pd g+\sum_{i=1}^{a-2}(\pd^iu)\circ g\pd^{a_1}g\cdots\pd ^{a_i}g,
  \eeq
  where $a_\ell\geq 1$ and $\sum a_l=a-2$.
   By \eqref{zproduct rule} with $a=1$ and \eqref{chain1}, we get
$$
|(\pd^{a-1}u)\circ g\pd g|_1\leq C(|u|_{a}(1+
C_{1/\e}\norm{g}_{D,1+\e}^{\frac{1}{1+\e}})+\norm{u}_{a-1}\znorm{g}_2)\leq
2C|u|_{a}(1+
C_{1/\e}\norm{g}_{D,1+\e}^{\frac{1}{1+\e}}).
$$

We now estimate the other terms  \eqref{pda1}. We remark that we do not
 have \eqref{holder_only} for Zygmund norms.
However, we will take advantage of  $1\leq i\leq a-2$.   By \eqref{chain rule}, \eqref{product rule}, and \eqref{holder_only} for H\"older norms, we get
\aln
&\norm{(\pd^iu)\circ g\pd^{a_1}g\cdots\pd ^{a_i}g}_{1}\leq \norm{(\pd^iu)\circ g}_1\norm{g}_{a_1}\cdots\norm{g}_{a_i}\\
&\qquad +\norm { u}_i\sum_{j=1}^i\norm{g}_{a_1}\cdots\norm{g}_{a_{j}+1}\cdots\norm{g}_{a_i}\\
&\qquad\leq C\left\{\norm{u}_{i+1+(a_1+\cdots+a_i-i-\e))} \norm{g}_{1+\e}+C
\norm{u}_1\norm{g}_{
a-1+\e}\right\},
\end{align*}
which gives us \eqref{zchain rule}.
\end{proof}
We need the following more general chain rule estimate. The proof can be found in the appendix  of ~\cite{MR3961327} for H\"older norms. The similar estimate can be obtained analogously by using the above chair rule, product rule
for the Zygmund spaces. We left the details to the reader.
\le{chain_rule}
Let $D_m$ be a sequence of Lipschitz domains in $\R^d$ 
of which $C_*(D_m)$ are bounded. Let $F_i = I + f_i$ map $D_i$ into $D_{i+1}$ with $\norm{f_i}_1\leq C_0$.
Then 
\ga\label{Chain_rule}
\norm{u\circ F_m\circ\cdots\circ F_1}_{D_0,r}\leq {C_r^m}\left\{\norm{u}_{r}+ \sum_i \norm{u}_1\norm{f_i}_r +\norm{u}_r\norm{f_i}_1        \right\},\  r\geq0;\\
\znorm{u\circ F_m\circ\cdots\circ F_1}_{D_0,r}\leq {C_r^m}\left\{\znorm{u}_{r}+ \sum_i \norm{u}_1\znorm{f_i}_r +C_{1/\e}\znorm{u}_r\norm{f_i}_{1 +\e} ^{\frac{1}{1+\e}}      \right\}, \  r>1.\nonumber
\end{gather}
\ele

We also need to extend an inverse mapping estimate in Webster \cite{MR999729} to the Zygmund spaces. Note that
$$
(\pd_x^ag)\circ F=\sum Q_\a(\pd f)\pd^{\a_1} f\cdots\pd^{\a_i}f,
$$
where $i\geq1, \a_j\geq1$,  $\a_1+\cdots+\a_i\leq a$, and $Q_\a(\pd f)$ are rational functions in $\pd f$ with $\norm {Q_\a(\pd f)}_{B_r,0}<C$.
\le{Webster}
Let $F= I+ f$ be a $C^1$ map from $B_r \defeq \{x\in \R^n :  \norm{x}\leq r\}\subset \R^n$ into $\R^n$ with\[
f(0)=0, \quad \norm{Df}_{B_r,0} \leq \th <\frac{1}{2}.
\]
Let $r'=(1-\theta)r$. Then the range of $F$ contains $B_{r'}$ and there exists a $C^1$ inverse $G= I+g$, which maps $B_{r'}$ injectively into $B_r$, with\[
g(0)=0, \quad \norm{Dg}_{B_{r'},0}\leq 2 \norm{Df}_{B_r,0}.
\]
 Assume further that
$f\in \Lambda^{a+1}(B_r)$. 
Then $g\in \Lambda^{a+1}(B_{r'})$ and
\ga
\norm{Dg}_{B_{r'}, a}\leq C_{a}\norm{Df}_{B_r,a},\quad a\geq0;\\
\znorm{Dg}_{B_{r'}, a}\leq C_{a}\znorm{Df}_{B_r,a}(1+C_{1/\e}\norm{f}_{1+\e}^{\frac{1}{1+\e}}),\quad a>1.
\end{gather}
\ele
In applications, the $r$ in the lemma will be bounded between two absolute constants. Thus the constants $C_a$ does not depend on $r, r'$. In fact, for convenience we will   drop the requirement that $f(0)=0$ replacing with condition that $f$ has compact support in $B_r$. This allows us to take $r'=r$, too.

\subsection{Estimates on the commutator}

For our application, we need to consider the commutator $[\dbar, E] \defeq \dbar E - E\dbar$.

\pr{commutator} Let $D$ be a bounded  $C^1$ domain in $\R^n$ and  $E$ be the Stein extension operator for $D$. Moreover, let $U = D + \eta\cdot \vec{N} $ where $\vec{N}$ is the outer unit normal vector of $D$ and
$
0<\eta< 1.
$ Then 
we have the following estimates
\ga
\norm{[\dbar, E]u}_{U,a}\leq  C_b \eta^{b-a-1}\norm{u}_{D, b},\quad b-1\geq a\geq1;\\
\znorm{[\dbar, E]u}_{U,a}\leq  C_b \eta^{b-a-1}\znorm{u}_{D, b}, \qquad b-1\geq a>1.
\end{gather}
\epr
\begin{proof} 
First, let $k,l$ be integers such that $0\leq k\leq l$.
Notice that for any function $f \in C^l(U)$ that vanishes on $D$, we have the point-wise estimate
for the $k$-th derivatives
$$
|f^{(k)}(x)|\leq C_l\dist(x, \d D)^{l-k}\norm{f}_{U,l},\quad \forall x\in U\backslash \ov D.
$$
Indeed, fix any  $x\in U\backslash \ov D$. We may assume without loss of generality that $0 \in \d D$ and $|x| = \dist(0, x) = \dist(x, \d D)$.  Let $\g (t) = t x$, $0\leq t \leq 1$ be the line segment that connects $0 ,x$. Let $N = l-k$. Then
by the fundamental theorem of calculus,
\eq{Nfund}
f^{(k)}(x) = \int_{0}^1\frac{d}{d t_1}\cdots \int_0^1\frac{d}{d t_{N}}f^{(k)}(t_N\cdots t_1 x) d t_{N}\cdots dt_1.
\eeq
Consequently,
$$
|f^{(k)}(x)| \leq C_N|x|^{N}\norm{f^{k+N}}_{U,0} \leq C_N\dist(x, \d D)^{l-k} \norm{f}_{U, l}.
$$
In particular,
$$
\norm{f}_{U,k} \leq C_{\ell-k}\eta^{\ell-k}\norm{f}_{U,\ell}.
$$
Now let $0\leq\a,\b<1$ and $k+\a\leq l+\b$. We claim that
\eq{ftc}
\norm{f}_{U,k,\a}\leq C_{l-k}\eta^{l-k+\b-\a}\norm{f}_{U,l,\b}.
\eeq
Indeed, assume first that $k=l$ and $\a\leq \b$. Then for any $|x-y|\leq \eta$, we have
\eq{fkxf}
\frac{|f^{(k)}(x)-f^{(k)}(y)|}{|x-y|^\a}\leq \norm{f}_{U,l,\b}|x-y|^{\beta-\a}
\leq \norm{f}_{U,l,\beta}\eta^{\b-\a}.
\eeq
If $|y-x|>c\eta$, we use $f^{(k)}|_{\ov D}=0$. We find $y^*\in\pd D$ such that $|y^*-y|=\dist(y,D)$. Then we get \eqref{fkxf} from
$$\frac{|f^{(k)}(x)-f^{(k)}(y)|}{|x-y|^\a}\leq C\eta^{-\a}(|f^{(k)}(x)|+|f^{(k)}(y)|)\leq C\eta^{-\a}\norm{f}_{\ell+\b}(|x|^\b+|y-y^*|^\beta).$$
Assume now that $l>k$. Note that any two points $x,y$ in $U$ can be connected by a smooth curve of length at most $C(D)|x-y|$. Putting the above together,
 we obtain
\al\label{fkxf+}
\frac{|f^{(k)}(x)-f^{(k)}(y)|}{|x-y|^\a}&\leq C(D) \norm{f}_{U,k+1}|x-y|^{1-\a}\leq C_{l-k}C(D)\norm{f}_{U,l}\eta^{l-k-\a}\\
&\leq C_{l-k}C(D)\norm{f}_{U,l+\b}\eta^{l-k-\a+\b}.
\nonumber
\end{align}

Finally, we apply $(\ref{ftc})$ to $f=[\dbar, E]u$ with $k+\a=a$, $l+\b=b-1$. Notice also that $\norm{[\dbar , E]u}_{U,b-1}\leq C_b \norm{u}_{D,b}$ where $C_b$ is the operator norm for $E : C^b (D)\to C^b(\R^n)$. Then
\eq{}
\norm{[\dbar, E]u}_{U,a}\leq  C_{b-a} C(D)\eta^{b-a-1}\norm{u}_{D, b}.
\eeq

\medskip
We now consider the Zygmund space case. Here we use interpolation on operator norms. Let $E_{U}$ be a Stein extension for functions on $U$.   We have
$$
E_{U}[\dbar,E_D]u=
E_{U}\dbar E_Du-E_UD\dbar u.
$$
Thus, we can write the H\"older estimates as
$$
\norm{E_U[\dbar,E_D]u}_{\C^n, a}\leq C\eta^{b-a-1}\norm{u}_b, \quad\forall u\in\Lambda^b(\C^n).
$$
We remark that the inequality is trivial when $b=a+1$. Thus we assume that $b>a+1$. We also have $a>0$.
We take non-integers $a_i, b_i$ satisfying  $b_0<b<b_1,0<a_0<a<a_1$, and $b_i>a_i+1$. Furthermore
$$
a=(a_0+a_1)/2, \quad b=(b_0+b_1)/2.
$$
We have $\znorm{[\dbar,E]u}_{a_i}\leq C_i\eta^{b_i-a_i-1}|u|_{b_i}$.
Since $u\to E_U[\dbar,E]u$ is a linear operator, we get via interpolation (see for instance~\cite{MR3961327})
$$
|[\dbar, E]u|_a\leq |E_U[\dbar, E]\widetilde u|_a  \leq C (C_0\eta^{b_0-a_0-1})^{1/2}(C_1\eta^{b_1-a_1-1})^{1/2}\znorm{\widetilde u}_b
$$
for any $\widetilde u\in\Lambda^b$ with $\widetilde u|_D=u$. This gives us the last inequality.
The proposition is proven.
\end{proof}

Let $U_0=D_0+t_0\cdot \vec{N}$ where $\vec{N}$ is the unit outer normal vector of the boundary. Fix $L\in \N$.  Moser constructed in \cite{MR0147741}  a smoothing operator $S_t: C^0(U_0)\to C^\infty(D_0)$,
\eq{smoothing_op}
S_t u(x) = \int_{|y|<1} u(y)\chi_t(x-y)dy ,\quad x\in D_0, \quad 0<t<t_0/ C,
\eeq
where $\int \chi(z) dz = 1$, $\chi_t(z) = \chi(z/t)$ and
\eq{orthog}
\int z^I \chi(z) dz = 0,
\quad 0<|I|\leq L.
\eeq
 Therefore, $S_t$ is a convolutional operator and
  for
 $0< t <
t_0/C$, we have
\ga\label{smoothing}
\norm{S_t u}_{D_0, a}\leq C_a t^{b-a}\norm{u}_{U_0,b},\quad 0\leq b\leq a <\infty;\\
\norm{(I-S_t)u}_{D_0, a}\leq C_b t^{b-a}\norm{u}_{U_0,b},\quad 0\leq a, \  0\leq b-a<L.
\end{gather}
Here the last inequality relies on \eqref{orthog}.
Via interpolation as explained in the proof of \rp{commutator} and applied to linear operator $S_t$ and $I-S_t$, we get from the above two inequalities the following for Zygmund norms:
\ga\label{Zsmoothing}
\znorm{S_t u}_{D_0, a}\leq C_a t^{b-a}\znorm{u}_{U_0,b},\quad 0< b\leq a <\infty; \\  \label{Zsmoothing1}
\znorm{(I-S_t)u}_{D_0, a}\leq C_
{
b
} t^{b-a}\znorm{u}_{U_0,b},   \quad a>0, \  0\leq b-a<L.
\end{gather}
 
\subsection{Stability of constants}
Finally, we recall the homotopy operators on $C^2$ strictly pseudoconvex domain constructed in \cite{MR3961327}. Let $D_0$ be a $C^2$ strictly pseudoconvex domain in $\C^n$ and $\U \subset \C^n$ be some open neighborhood of $\ov{D_0}$.   Then for any $\var\in \Lambda^r_{(0,1)}(\overline{D_0})$ with $r>1$, we have the homotopy formula
\eq{htf}
\var= \dbar P_{D_0,\U}\var+Q_{D_0,\U}\dbar \var, \quad \text{on $D_0$}
\eeq
where\[
P_{D_0,\U}\var(z)= \int_{\zeta\in\U} \Om^{0}_{0,0}(\zeta,z)\we E_{D_0}\var(\zeta) + \int_{\zeta\in\U\backslash D_0} \Om^{01}_{0,0}(\zeta,z)\we [\dbar, E_{D_0}]\var(\zeta).
\]
Here $\Om^{0}_{0,0}(\zeta,z), \Om^{01}_{0,0}(\zeta,z)$ are forms
 of types $(0,0)$ and $(0,1)$ in $z$, respectively. Moreover, we have the following estimate
\eq{1/2 estimate}
\znorm{P_{D_0, \U}\var}_{D_0,r+1/2}\leq C_r(
D_0)\th_0^{-r-\mu}\znorm{\var}_{D_0,r}, \quad r > 1
\eeq
where $\th_0 = \dist(D_0, 
\d
\U)$, and
 $\mu$ is some constant depending only on the dimension and $C(D_0) > 0$ is another constant depending on $C^2$ norm of the defining function.  Similar formula and estimates hold with $Q$ in place of $P$ and $\var$ a $(0,2)$ form.  We refer to  \cite{MR3961327} for more details on these operators and the estimates.

In our application,
we will also apply estimates on $P, Q$  to a sequence of domains $D_j$ such that $\dist(D_j,\pd \U)$ are bounded below by a fixed positive
number depending on initial domain $D_0$.
 Consequently, we can absorb $\th_0^{-s-\mu}$ into the coefficient in the estimate. We shall also drop the subscript in $D_0$ for simplicity if no confusion would be caused.
We remark that the constant $C_s$ in \eqref{1/2 estimate} depends on $s$ and it may not be bounded as $s$ tends some special values such as a positive integer.

We make a remark about the stability of a constant under $C^2$ perturbation of the domains.
\rem{\label{stremark}
Let $D_0 \defeq \{x \in \U: \rho_0(x)<0\} \subset \U\subset \R^n$ be a domain with $C^2$ boundary where $\U$ is some open neighborhood of $D_0$ and $\rho_0$ is a 
(standard)
 $C^2$   defining function. Let
\[
\mathcal{D}_{\e_0} = \{\rho\in C^2(\U): \norm{\rho - \rho_0}_{\U,2}\leq \e_0\}.
\] Here $\e_0>0$ is sufficiently small such that for all $\rho\in \mathcal{D}_{\e_0}$, we have  $d\rho(x)\ne 0 $ on $\{x\in \U: \rho(x) = 0\}$.

Suppose that there is a function
\[
\mathcal{C} : D_{\e_0} \to (0,\infty).
\]
We say $\mathcal{C}$ is {\emph{upper stable}} (resp. \emph{lower stable}) under small $C^2$ perturbation of $\rho_0$ if there exists 
$\e(D_0)>0$
 and a 
 constant $C_0(D_0)>
1$ possibly dependent of $\rho_0$, such that
\eq{stability}
\mathcal{C}(\rho) \leq C_0
( 
D_0)
\mathcal{C}(
\rho_0) \quad (\text{resp.}\ \ \mathcal{C}(\rho_0) \leq C_0(D_0)\mathcal{C}(\rho))
\eeq
for all
$\rho$ satisfying
$\norm{\rho - \rho_0}_{\U,2}\leq \e(D_0)$.

The following are examples of upper stable mappings that will be used for our purposes.
\begin{enumerate}
\item
Recall that we introduce {\em standard} defining functions for $C^k$ domains with $k\geq1$ in Section~\ref{Local_NN}. By being standard, we mean that  the defining functions depend only on the domains in construction. 
We will write $\mathcal C(\rho_0)$ when  $\rho_0$ is a standard definition function of $D_0$. 
There are other ways to construct standard definition functions.  For instance, we can replace $\rho_0$ by a Whitney extension of $\rho_0|_{\ov D_0}$ so that $\rho_0\in C^\infty$ away from $\ov D_0$.
There are other ways to construct the definition functions. For instance the Stein extension 
 can also be used. 
\item The operator norms of Stein extension operator between $\Lambda^r(D_0)\to \Lambda^r(D_0)$ for some $r>1$ is stable under small $C^2$ perturbation of the domain $D_0$.  Indeed, it is well-known that the operator norm $C_r(D_0)$ only depends on the Lipschitz constant of $D_0$. Thus $C_r(\widetilde D)<C_0C_r(D_0)$ when $\d \widetilde D$ has a $C^1$ defining function $\widetilde\rho$ with 
    $\|\widetilde\rho-\rho_0\|_1<\e(D_0)$ sufficiently small for some constant $C_0(D_0, \e)$.

\item The constants in  the estimates  $(\ref{convexity}), (\ref{product rule}), (\ref{chain rule}), (\ref{zconvexity}), (\ref{zproduct rule}), (\ref{zchain rule})$ and 
\rl{chain_rule}
are also stable under small $C^2$ perturbation of $D$ provided that $D$ is a $C^1$ domain.
This should follow in principle the proofs of these inequalities. Alternatively, it also follows from the above remark.  Indeed, let $\widetilde D = \{z\in \U: \widetilde\rho(z)< 0\}$ where $\norm{\widetilde\rho - \rho}_{\U,2}\leq \e$ for some the some  $\e(\rho) > 0$. Let $E_{\widetilde D}$ be the Stein extension operator on $\widetilde D$. Then we have\[
\begin{aligned}
\norm{u}_{\widetilde D, l}\leq \norm{E_{\widetilde D} u}_{\U, l}&\leq C_{a,b}(\U) \norm{E_{\widetilde D} u}^\th_{\U,a}\norm{E_{\widetilde D}u}^{1-\th}_{\U,b}\\
&\leq C_{a,b}(\U)C'_a(\widetilde D)C^{''}_b(\widetilde D) \norm{u}^\th_{\widetilde D,a}\norm{u}^{1-\th}_{\widetilde D,b}
\end{aligned}
\]
for $l=\theta a+(1-\theta)b$.
Consequently, the stability of $$C_{a,b}(\widetilde D) = C_{a,b}(\U)C'_a(\widetilde D)C^{''}_b(\widetilde D)$$ follows from the stability of $C_a'(\widetilde D), C^{''}_b(\widetilde D)$ of Stein extension operator.  The proofs for $(\ref{product rule}), (\ref{chain rule}), (\ref{zconvexity})$, $(\ref{zproduct rule})$ and $(\ref{zchain rule})$ are similar. For the coefficients in Lemma $\ref{chain_rule}$, we simply notice they are finite products of the constants from $(\ref{convexity}), (\ref{product rule}), (\ref{chain rule})$ and some dimensional constants. Consequently, they are also stable under small $C^2$ perturbation.

\item It is easy to see that the operator norms for Nash-Moser smoothing operator \eqref{smoothing_op} are stable under small $C^2$ perturbation of the domain $D_0$.
\item The operator norms for $(\ref{htf})$ in (\ref{1/2 estimate}) are stable under small $C^2$ perturbation of the domain $D_0$. This was  proved
      in~\cite{MR3961327}*{Theorem $5.2$}.
\end{enumerate}

The
 upper 
 stability of these constants is important  for the convergence of our iteration process in Section~\ref{sec iteration} and also for the proof of the {\it lower}
  stability of $\de_r(D_0)$. Notice that the latter condition already played an important role in the proof of
Theorem~\ref{Local NN}.
}

\section{Approximate solution via homotopy formula }\label{approximation_of_embedding}

\setcounter{thm}{0}\setcounter{equation}{0}

Let $D_0$ be a strictly pseudoconvex domain in $\C^n$ with $C^2$ boundary. Given the initial integrable almost complex structure $X_{\overline\a}=\d_{\overline\a}+A_{\overline\a}^\b \d_\b$ on $\overline D_0$, we wish to find a transformation defined on $\ov D_0$ to transform the complex structure into a new complex structure closer to the standard complex structure while $\ov D_0$ is transformed to a new domain that is still $C^2$ strictly pseudoconvex.

According to Lemma \ref{elem_pp}, after a perturbation of the form $F= I+ f$ with $Df$ small, the new structure $\{\d_{\ov\a}+\hat A_{\ov\a}^\b\d_\b\}$ has the matrix form
\eq{}
\hat A \circ F = (I+ \ov{\d f}+A\dbar f)^{-1}(A+\dbar f + A\d f).
\eeq

We first formally  decide the correction $f$ following Webster \cites{MR999729}. Then we indicate the obstructions and make necessary modifications.

From now on, we shall regard $A_{\overline\a}^\b$ as the coefficients of  $(0,1)$ forms by simply identifying $A_{\overline\a}^\b\d_\b$ with $A^\beta:=A_{\overline\a}^\b dz^{\overline\a}$, where $\b=1,\dots, n$. 
 We can then apply the  homotopy formula \eqref{htf} componentwise to $A:=(A^1,\dots, A^n)$ and write
 $$A = \dbar P A+ Q\dbar A.
 $$
 For  Newton's method, we would take $f= -P A$. Then
\eq{}
A+\dbar f +A\d f= \dbar P A + Q\dbar A - \dbar PA +A\d f = Q\dbar A +A\d f.
\eeq
Using the integrability condition $\dbar A = [A, \d A]$ and product rule (\ref{zproduct rule}), formally we would have $\znorm{A'}\leq \znorm{A}^2$. This is the correction used in Webster's proof of the classical Newlander-Nirenberg theorem \cite{MR999729}. However, similar to Webster \cite{MR1128608}, Gong-Webster \cite{MR2868966} and Gong \cite{MR3961327}, the homotopy operator $P$ does not gain the full derivative lost in applying $\dbar$ to $A$. Therefore, we need to apply a smoothing operator to $-PA$ so that the iteration does not terminate within finitely many steps. Note also that the transformation $F$ must be defined on $\ov{D_0}$. Consequently, we need to use Nash-Moser smoothing method in a way different from the above mentioned work. Namely, we  first extend $PA$ to a larger domain via Stein extension operator  and then apply the smoothing operator $S_t$. This ensures that the new complex structure is defined on the closure $\ov D_1 = F(\ov D_0)$ where the new structure is still formally integrable
 and has the same regularity as the original complex structure. Therefore, we modify $f=-PA$ and take
\[
f=-S_tE_{D_0}P_{D_0,U_1}A.
\]
Here we assume that 
\eq{distDU}
\dist(D_0,\pd B_0)\geq c_0^*
\eeq
and via a cut-off function, we assume that $E_{D_0}u$ has compact support in
 $$
 B_0=\{z\in\C^n,\quad |z|<\sigma_0\}.
 $$
   Note that $f$ still have compact support, provided
\eq{c0ss}
t<c_0^{**}.
\eeq
Consequently, we have the following identity on $D_0$:
\begin{align}
A+\dbar f + A \d f&=A- \dbar S_t  EPA +A\d f
\\&= A- S_t\dbar  E P A+[S_t, \dbar]EPA+A\d f
\nonumber \\&= A- S_t E\dbar PA +S_t [E, \dbar] PA + A\d f
\nonumber\\&= A- S_t E A  + S_t E Q\dbar A +S_t [E, \dbar] PA + A\d f
\nonumber\\&= (I-S_t)EA + S_t E Q\dbar A +S_t [E, \dbar] PA + A\d f,
\nonumber\end{align}
where in the third equality, we use $[S_t, \dbar]=0$ on $D_0$ when acting on $C^1(U_0)$.

According to the above computation and Lemma \ref{elem_pp}, our new error $\hat A$ satisfies
\eq{five_term}
\hat A\circ F = (I+\overline{\d f}+ A\dbar f)^{-1} \left\{(I-S_t)EA + S_t E Q\dbar A +S_t [E, \dbar] PA + A\d f\right\}.
\eeq
We shall denote
\ga{}\label{five_term1}
I_1 = (I-S_t)EA, \quad I_2= S_tEQ\dbar A, \\
\label{five_term2}
I_3= -S_t[\dbar,E]PA,\quad I_4=A\d f, \quad
I_5 = \overline{\d f}+A\d f, \\ \label{five_term3}
\widetilde A = (I+I_5)^{-1}(I_1+I_2+ I _3 +I_4).
\end{gather}
 Then we have
 \eq{}
 \hat A\circ F= \widetilde A, \ \text{on $\ov D_0$}; \quad
 \hat A = \widetilde A\circ G, \ \text{on $\ov D_1$.}
  \eeq
 Here $D_1:=F(D_0)$ and  $G = F^{-1}$ maps $D_1$ onto $D_0$.

Before proceeding, let us briefly discuss the plan for proving Theorem \ref{Global NN}.  In Section \ref{sec  correction estimate}, we estimate lower order norm $\znorm{f}_{D_0,s}$ for some $s>2$ of the transformation $F=I+f$ and give a rough estimate of   the new complex structure $\hat A$ on the closure of the new domain $D_1$. In Section~\ref{four term estimate}, we refine our estimates on the lower and high order norms $|\widetilde A|_{D_0, s}, |\widetilde A|_{D_0, r}$ by estimating $I_1,\ldots, I_5$.
  In Section~\ref{sec  iteration}, we describe the iteration scheme and verify the induction hypotheses. We shall obtain uniform control for the gradient, second order derivatives of $f$, and the Levi form of the defining function for iterated domains.  In Section \ref{convergence}, we run the iteration and determine all the parameters in order to achieve optimal regularity result. Finally, we show the convergence of the composition of a sequence of transformations on $\ov D_0$ in $\Lambda^k$ norm for suitable $k$.

\section{Change of coordinates and new complex structure}\label{sec correction estimate}
\setcounter{thm}{0}\setcounter{equation}{0}

Let $A\in \Lambda^r(\ov{D_0})$ be the error term in the original almost complex structure where $1 < r < \infty$.
Let $\frac{3}{2} < m \leq r+ \frac{1}{2}$ and $1 < \ell \leq r$.

We start by deriving the following two estimates for $f=F-I$: 
\al\label{correctionest}
\znorm{f}_{D_0, m}&=
\znorm{S_t EPA}_{D_0,m}\leq C'_m \znorm{EPA}_{U_0,m}\leq C'_mC''_m \znorm{PA}_{D_0,m}\\
&\nonumber\leq C'_mC''_mC_m''' \znorm{A}_{D_0, m-1/2},
 \\
\znorm{f}_{D_0, \ell+1}& =
\znorm{S_t EPA}_{D_0, \ell+1}\leq C'_{\ell+1}t^{-\frac{1}{2}}\znorm{EPA}_{U_0, \ell+1/2}\label{correctionest1}\\
&\nonumber\leq C'_{\ell+1}C''_{\ell+1}C'''_{\ell+\frac{1}{2}}t^{-\frac{1}{2}}\znorm{A}_{D_0,l}, \quad
\end{align}
where $C'_\bullet$ is the constant from the Nash-Moser smoothing operator (\ref{smoothing}) which is independent of the domain $D_0$, $C''_{\bullet}$ is the constant from Stein's extension operator (\ref{extension}) and $C'''_{\bullet}$ is the constant from  estimate (\ref{1/2 estimate}). Constants in the second estimate have similar meaning.

Let us first describe how we control the norms in iteration. Let $3/2<s<3$.  We need to get rapid convergence in low order derivatives of $f$. This will be measured by the $s$-norn $\znorm{A_0}_{D_0,s}$.  There are two  estimates \eqref{correctionest} and
\eqref{correctionest1} which are available  to control the lower order derivatives of $f$. We will use \eqref{correctionest} to control the second-order derivatives of $f$ and thus the Levi-forms of the domains in iteration. We will use \eqref{correctionest1} to control the $(s+1)$-norm of $f$.  Let $\s_0 > 0$  be any number large enough such that 
$$
\ov{D_0} \subset \mathcal{U} \subset B_0,
$$
where $\U$ is an open neighborhood of $\ov{D_0}$ and $B_0 = \{z\in\C^n: |z| < \s_0\}$. We shall still denote by $f$ the extension $Ef$ to $B_0$ where $E$ is the Stein extension operator. We may assume that $Ef$ has compact support in $B_0$.

To simplify our notation,  we denote by $C_m(D_0)$ finite products of {\emph{upper stable}} constants.  Notice that by definition of upper stability, finite products of upper stable constants is still upper stable. We will then use $C^*_m(D_0)$ to indicate in the context when these constants are fixed for the rest of the paper. 

According to \eqref{correctionest} and
\eqref{correctionest1}, we have
\begin{gather}
\label{initial_est}
\norm{f}_{B_0, 2}\leq C^*_2\znorm{A}_{D_0,s},\\ 
\label{initial_est1}
\znorm{f}_{B_0, s+1}\leq
C^{*}_{s} t^{-1/2}\znorm{A}_{D_0,s}.
\end{gather}
Let us first  assume that 
\eq{ini_condition}
 |A|_{D_0, s} < \f{
1}{C_s^{**}}, \quad C_s^{**}>
\sqrt{8n}C_2^*.
\eeq

Here we have fixed $C_2^*, C_{s}^*$ and we will adjust the constant $C_{s}^{**}$ a few times, which will be indicated sometime for clarity. By \eqref{initial_est} and \eqref{ini_condition}, we have
\eq{fnorm2} 
\norm{Df}_{B_0, 0} \leq \sqrt{2n}\norm{f}_{B_0,2}\leq\f{\sqrt{2n} C_2^*}{C_s^{**}} 
< \frac{1}{2}.
\eeq
Thus \rl{Webster} gives us for $F=I+f$. 

Recall that $f$ has compact support in $B_0$. Therefore $F$ is a diffeomorphism from $B_0$ onto itself.  Let $F^{-1} = G$ be its inverse mapping defined on $B_0$. 
The estimate \eqref{fnorm2} also ensures that the constants in Stein extension for $F(D_0)$ and convexity of norms are equivalent to the constants for $D_0$. However, in next section we will impose a stronger condition ensuring $F(D_0)$ remains strictly pseudoconvex.

From \eqref{five_term}-\eqref{five_term3}, we have
\begin{equation}
\hat A\circ F = (I+I_5)^{-1}(I_1+I_2+I_3+I_4)=:\widetilde A
\end{equation}
where $I_1 = (I-S_t)EA$, $I_2= S_tEQ\dbar A$, $I_3= -S_t[\dbar,E]PA$, $I_4=A\d f$ and $I_5 = \overline{\d f}+A\d f$.

We wish to estimate $\znorm{\hat A}_{D_1, \ell}$ in terms of $\Lambda^s$ and $\Lambda^r$ norms of $f, A$. We do this by first applying the chain rule to $\hat A = \widetilde A\circ G$  and reduce the problem to 
estimating $\znorm{\widetilde A}_{D_0,\ell}$. Then we use the convexity of H\"older norms to further reduce the problem to estimating $\znorm{\widetilde A}_{D_0,r}, \znorm{\widetilde A}_{D_0,s}$.

According to \eqref{initial_est}, \eqref{initial_est1}, \eqref{fnorm2} 
and Lemma~\ref{Webster}, 
we have
\ga\label{estgs}
\norm{g}_{B_{0},2}\leq C_2 \norm{f}_{B_0,2}\leq C_2 \znorm{A}_{D_0,s}\\
\label{estgss}
\znorm{g}_{
B_0, 
s+1}\leq C_s \znorm{f}_{B_0, s+1}\leq
C_{s} t^{-1/2}\znorm{A}_{D_0,s}.
\end{gather}

By 
\eqref{correctionest},  \eqref{fnorm2}, and Lemma~\ref{Webster},   we obtain
\eq{estgs+}
\znorm{g}_{D_{1},
m}\leq C_m\znorm{f}_{D_0,m}\leq C_m\znorm{A}_{D_0,m-1/2}.
\end{equation}

Let $D_1 = F(D_0)$. We apply chain rule estimate (\ref{zchain rule}) to $\hat A$ on $D_1$  together with \eqref{estgs} and \eqref{ini_condition}  to obtain
\begin{align}
|\hat A|_{D_1,m}=|\widetilde A\circ G|_{D_1, m}&\leq C_m(|\widetilde A|_{D_0,m}+ 
\|\widetilde A|_{D_0,1}\cdot \znorm{g}_{D_1,m}).
\end{align}
Using 
\eqref{estgs+}, we have
\begin{equation}\label{Error}
|\hat A|_{D_1,
m}\leq C_m |\widetilde A|_{D_0,m}.
\end{equation}
Therefore, it suffices to estimate $\widetilde A$.

According to the convexity of H\"older-Zygmund norms (\ref{zconvexity}) with $a=s, b=r$ and $\ell = (1-\th) s +\th r$ for some $0\leq \th\leq 1$, we have the following estimates for intermediate derivatives,
\eq{}
|\widetilde A|_{D_0, \ell}\leq 
C_{r,s,\ell} |\widetilde A|^{1-\th}_{D_0, s}|\widetilde A|^\th_{D_0, r}.
\eeq
Consequently, we can reduce the problem of estimating the intermediate derivatives to estimating $|\widetilde A|_{D_0, s}$ and $|\widetilde A|_{D_0,r}$, which we shall often refer to as low and high order derivative estimates.

To this end, we apply product rule (\ref{zproduct rule}) to $m$-norm of $\widetilde A = \hat A \circ F$. We get
\begin{align}
|\widetilde A|_{D_0,m}&\leq C_m\cdot\left ( \znorm{(I+I_5)^{-1}}_{D_0, m}\cdot \sum_{i=1}^4 \norm{I_i}_{D_0,0}+\norm{(I+I_5)^{-1}}_{D_0,0}\cdot \sum_{i=1}^4\znorm{I_i}_{D_0,m}\right ) \\ \nonumber
& \leq C_m\cdot\left ( \znorm{(I+I_5)^{-1}}_{D_0, m}\cdot \sum_{i=1}^4 \znorm{I_i}_{D_0,s}+\znorm{(I+I_5)^{-1}}_{D_0,s}\cdot \sum_{i=1}^4\znorm{I_i}_{D_0,m}\right).
\end{align}
When $m = s$, we have the low order estimate
\begin{align}\label{low}
|\widetilde A|_{D_0,s} \leq 2C_s   \znorm{(I+I_5)^{-1}}_{D_0, s}\cdot \sum_{i=1}^4 \znorm{I_i}_{D_0,s}.
\end{align}
Similarly when $m = r$, we have the high order estimate
\eq{high}
|\widetilde A|_{D_0,r} \leq C_r\cdot\left (\znorm{(I+I_5)^{-1}}_{D_0, r}\cdot \sum_{i=1}^4 \znorm{I_i}_{D_0,s}
+\znorm{(I+I_5)^{-1}}_{D_0,s}\cdot \sum_{i=1}^4\znorm{I_i}_{D_0,r}\right).
\eeq
We shall begin to estimate the right hand sides in the next section.

\section{Estimate of $I_1,\cdots, I_5$ and $\hat A$}\label{four term estimate}
\setcounter{thm}{0}\setcounter{equation}{0}

Let $A \in \Lambda^r$, $\frac{3}{2} < r \leq \infty$ be the error term in the original complex structure. 
Let $\frac{3}{2} < s < 3$.  In this section, we assume that $\frac{3}{2} < m \leq r$ if  $r <\infty$ and $\frac{3}{2} < m <\infty$ if $A\in C^\infty$. 
We  also replace the initial condition \eqref{ini_condition} by a stronger condition
\eq{ini_condition2}
t^{-1/2}\znorm{A}_{D_0,s}\leq\frac{1}{C_s^{**}}
\eeq
where $t \in (0, 1)$ and $C_s^{**}$ will be adjusted several times in this section.

When $r <\infty$, we choose our smoothing operator $S_t$ depending on $r$. More precisely, in \eqref{smoothing_op}, we choose $L = \lceil r \rceil$  where $\lceil \cdot \rceil$ means rounding up to the closest integer.  By the property of the smoothing operator (\ref{Zsmoothing1}), we get 
\eq{I1_m}
\znorm{I_1}_{D_0,m}=\znorm{(I-S_t)EA}_{D_0,m}\leq C_{r}t^{r-m}\znorm{EA}_{U_0,r}\leq C_r t^{r-m}\znorm{A}_{D_0, r}.
\eeq
Substituting $s, r$ into $m$, we have
\ga\label{I1_s}
\znorm{I_1}_{D_0,s}\leq C_r t^{r-s}\znorm{A}_{D_0,s}, \\
\znorm{I_1}_{D_0,r}\leq C_r \znorm{A}_{D_0,r}\label{I1_r}
\end{gather}
where the former inequality will be used to prove the rapid convergence of \eqref{low} and the latter will be used to control the growth rate of \eqref{high}. 

When $r = \infty$, we construct another smoothing operator by choosing $L = 3$ in the construction of  $S_t$.  Then for all $\frac{3}{2} < m < \infty$, we have  
\ga
|I_1|_{D_0,m} = \znorm{(I-S_t)EA}_{D_0,m}\leq C_{m}\znorm{A} _{D_0,m}. \label{I1_m_inf}
\end{gather}

To estimate $I_2$, we need to use the integrability condition $\dbar A = [A, \d A]$ on $D_0$ and estimate~\eqref{correctionest}.
We will also apply \eqref{1/2 estimate} to  $\dbar A\in \Lambda^{s-1}(D_0)$ in the following estimate, which forces us to impose a stronger condition
 \eq{}
 s>2.
 \eeq
By $\db A=[A,\pd A]$, we have 
\begin{align}\label{I2_m}
\znorm{I_2}_{D_0,m}&=\znorm{S_t EQ\dbar A}_{D_0,m}\leq C_s' t^{-1/2}\znorm{EQ\dbar A}_{U_0, m-1/2}
\\
&\leq C_m'C_{m-\frac{1}{2}}''t^{-1/2}\znorm{Q\dbar A}_{D_0,m-1/2}
\nonumber
\\
&\nonumber
\leq  C_m'C_{m-\frac{1}{2}}'' C_{m-\frac{1}{2}} t^{-1/2}\znorm{\dbar A}_{D_0, m-1}
\\&\nonumber
\leq C_m t^{-1/2}(\znorm{A}_{D_0, m-1}
\norm{A}_{D_0, 1} + \znorm{A}_{D_0, m}
\norm{A}_{D_0, 0})\\
&\nonumber
\leq C_m t^{-1/2}\znorm{A}_{D_0,s}
\znorm{A}_{D_0,m}.
\end{align}
Using initial condition \eqref{ini_condition2} and applying the above estimate with $s, r$ in place of $m$, we get
\ga
\label{I2_s}
\znorm{I_2}_{D_0,s}\leq C_st^{-1/2} \znorm{A}^2_{D_0,s}, \\
\znorm{I_2}_{D_0,r}\leq  C_r \znorm{A}_{D_0,r}.\label{I2_r}
\end{gather}
When $r = \infty$, we have
\ga
\znorm{I_2}_{D_0,m}\leq C_m\znorm{A}_{D_0, m}. \label{I2_m_inf}
\end{gather}

For the estimate of $I_3$, we have 
\begin{align}\label{I3_m}
\znorm{I_3}_{D_0,m}=\znorm{S_t[\dbar, E]PA}_{D_0,m}&\leq C'_m \znorm{[\dbar, E]PA}_{U_0,m}\\
&\nonumber
\leq C''_{r,m} t^{r+1/2-m-1}\znorm{PA}_{D_0,r+1/2}\quad
\\
&\nonumber
\leq C_r t^{r-m-1/2}\znorm{A}_{D_0, r},
\end{align}
where we used Lemma \ref{commutator} in the second inequality. 
Applying the above estimate with $s, r$ in place of $m$,  we have
\ga
 \znorm{I_3}_{D_0,s}
 \leq C_r t^{r-s-1/2}\znorm{A}_{D_0,s}   \quad r-\frac{1}{2}\geq s > 2 \label{I3_s}, \\
\znorm{I_3}_{D_0,r}\leq C_r t^{-1/2}\znorm{A}_{D_0,r}. \label{I3_r}
\end{gather} 
We remark here that $t^{-1/2}$ in the coefficient of \eqref{I3_r} is the main obstruction that prevents us from having a linear growth in~\cites{MR995504, MR3961327} for the high order norms.

When $r=\infty$, we have 
\ga
\znorm{I_3}_{D_0,m}\leq C_m t^{-1/2}\znorm{A}_{D_0,m}. \label{I3_m_inf}
\end{gather} 

The estimate for $I_4$ is more involved.
Recall that $f = -S_t EPA 
$. By
{\eqref{initial_est1}}
we have  
\begin{align}\label{I4_m}
\znorm{I_4}_{D_0,
m}=\znorm{A\d f}_{D_0,m}&\leq {C}_{m+1}\left(\znorm{A}_{D_0,m}
\norm{f}_{D_0,1}
+\norm{A}_{D_0,0}\znorm{f}_{D_0,m+1}\right)\\
&\nonumber\leq C_{m+1}\left(\znorm{A}_{D_0,m}
\znorm{A}_{D_0,s}+\norm{A}_{D_0,0}t^{-1/2}\znorm{A}_{D_0,m}\right)\\
&\nonumber\leq C_{m+1}t^{-1/2}
\znorm{A}_{D_0,s}
\znorm{A}_{D_0,m}.
\end{align}
Using initial condition \eqref{ini_condition2} and applying the above estimates with $r, s$ in place of 
$m$ respectively, we get
\ga
\label{I4_s}
\znorm{I_4}_{D_0,s} 
\leq C_{s+1}t^{-1/2}\znorm{A}_{D_0,s}^2, \\
\znorm{I_4}_{D_0,r}\leq C_{r+1} \znorm{A}_{D_0,r}.\label{I4_r}
\end{gather}
When $r = \infty$, we have
\ga
\znorm{I_4}_{D_0,m}\leq C_{m+1} \znorm{A}_{D_0,m}.\label{I4_m_inf}
\end{gather}

Finally, we need to estimate the low and high order derivatives for $I_5$ and $(I+I_5)^{-1}$:
\eq{I5_m}
\znorm{I_5}_{D_0,m} = \znorm{\overline{\d f}+ A\dbar f}_{D_0,m}\leq C_{m+1} t^{-1/2}\znorm{A}_{D_0,m}.
\eeq

Using the initial condition \eqref{ini_condition2} for $C_{s}^{**}$ sufficiently large and applying the above estimate with $r$ in place of  $m$, we get
\ga
\label{I5_s}
\znorm{I_5}_{D_0,s}\leq
1/C_s, \\
\znorm{I_5}_{D_0,r} \leq C_{r+1} t^{-1/2}\znorm{A}_{r}. \label{I5_r}
\end{gather}
When $r = \infty$, 
\ga
\znorm{I_5}_{D_0,m} \leq C_{m+1} t^{-1/2}\znorm{A}_{m}. \label{I5_m_inf}
\end{gather}

We 
now consider $(I+I_5)^{-1}-I$. By matrix inversion formula  $(I+I_5)^{-1}= \det(I+I_5)^{-1}(A_{ij})$ where $(A_{ij})$ is the transpose of the adjugate matrix of $I+I_5$. Notice that every entry in $(I+I_5)^{-1}-I$ is a polynomial 
in $(\det (I+I_5))^{-1}$ and entries of $I_5$  without constant term and with fixed degree.

We can now estimate $\znorm{(I+I_5)^{-1}-I}_{D_0,m}$ using product rule
$$
\znorm{\Pi_{i=1}^n u_i}_{D_0, m}\leq C\sum_{i=1}^n \znorm{u_i}_{D_0,m}\prod_{i\ne j}\norm{u_j}_{D_0,0}.
$$
It is then easy to show that
\begin{equation}
\znorm{(I+I_5)^{-1}-I}_{D_0, m}\leq C_m \frac{\znorm{I_5}_{D_0, m}}{1-C_*\norm{I_5}_{D_0,0}}, \quad C_*>0.
\end{equation}
Consequently, 
\eq{I+I5_inv_m}
\znorm{(I+I_5)^{-1}}_{D_0,m}\leq C_m(1 + \znorm{I_5}_{D_0,m}).
\eeq
By \eqref{I5_s},
we have the estimate
\eq{I+I5_inv_s}
\znorm{(I+I_5)^{-1}}_{D_0,s}\leq C_s(1 + \znorm{I_5}_{D_0,s})\leq2 C_s.
\eeq
Letting $m = r$, we have
\eq{I+I5_inv_r}
\znorm{(I+I_5)^{-1}}_{D_0,r}\leq C_r(1 + \znorm{I_5}_{D_0,r}).
\eeq
Similarly, when $r=\infty$, 
\ga
\znorm{(I+I_5)^{-1}}_{D_0,m}\leq C_m(1 + \znorm{I_5}_{D_0,m}). \label{I+I5_inv_m_inf}
\end{gather}

Using s-norm derivative estimates  \eqref{I1_s},\eqref{I2_s},\eqref{I3_s},\eqref{I4_s} and \eqref{I+I5_inv_s} in the product rule formula  \eqref{low} and the fact that $\norm{g}_{1+\e}<\frac{1}{2}$, we obtain
\begin{equation}
|\widetilde A|_{D_0,s} \leq C_{r}\left( t^{r-s-1/2}\znorm{A}_{D_0,r}+t^{-1/2}\znorm{A}^2_{D_0,s}\right).
\end{equation}

Let $C_s^{**}$ be sufficiently large in the initial condition \eqref{ini_condition2}, and we may assume that $s$-norms of $I_1, I_2, I_3, I_4, (I+I_5)^{-1}$ are uniformly bounded by some positive constant $C_r$. 

Then by using $r$-norm estimates  \eqref{I1_r},\eqref{I2_r},\eqref{I3_r},\eqref{I4_r} and \eqref{I+I5_inv_r}  in the  product rule formula \eqref{high}, we obtain
\eq{}
|\widetilde A|_{D_0, r}\leq C_r t^{-1/2}\znorm{A}_{D_0,r}, \quad r > 2.
\eeq
Similarly, when $r=\infty$,  we have
\ga
|\widetilde A|_{D_0, m}\leq C_m t^{-1/2}\znorm{A}_{D_0,m}, \quad m > 2.
\end{gather}

Noticing that $\hat A = \widetilde A \circ G$,  we apply \eqref{estgs+} and \eqref{Error}
to get
\ga
\label{Low}
|\hat A|_{D_1,s} \leq C^*_{r}\left( t^{r-s-1/2}\znorm{A}_{D_0,r}+t^{-1/2}\znorm{A}_{D_0,s}^2\right), \quad r-\frac{1}{2}\geq s > 2, \\
|\hat A|_{D_1,r}\leq C^*_r t^{-1/2}\znorm{A}_{D_0,r}, \quad r > 2. \label{High}
\end{gather}
And when $r=\infty$,
\ga
|\hat A|_{D_1,m}\leq C^*_m t^{-1/2}\znorm{A}_{D_0,m}, \quad m> 2. \label{High_inf}
\end{gather}


We have  derived the estimates for new $A$ and $f,g$ under the assumption \eqref{ini_condition2} where the $C_{s}^{**}$ is now fixed for the rest of the proof. Moreover, $C_2^*, C_{s}^*$ have been fixed in \eqref{initial_est}, \eqref{initial_est1} and $C_r^*, C_m^* $ have been fixed in \eqref{Low}, \eqref{High} and \eqref{High_inf}.
\section{Levi form of iterated domains}\label{sec  iteration}
\setcounter{thm}{0}\setcounter{equation}{0}
Let us summarize what we have achieved so far 
under the assumption \eqref{ini_condition2} on error $A_0$.
Let $D_0$ be a strictly pseudoconvex domain with $C^2$ boundary in $\C^n$ and let $X_{(0)} = \dbar + A_0 \d\in \Lambda^r(\ov{D_0})$ be the initial perturbed integrable almost complex structure where $\dbar=(\d_{\ov 1},\dots, \d_{\ov n})$ is the standard complex structure on $\C^n$ and $\d$ is its conjugate.  In Section \ref{approximation_of_embedding}, we defined  $F_0 = I + f_0$ to be our first approximate solution where $$
f_0 = -S_{t_0}E_{D_0} P_{D_0} A_0 
$$ for some $t_0 > 0$ to be determined.  Let $D_1 \defeq F_0(D_0)$  be the new domain and $A_1$ be the error for the new almost complex structure on $D_1$ where 
$$A_{1} \circ F_0= (I +\ov  \d f_0
+ A_0 \d  \overline f_0)^{-1}(A_0 + \dbar f_0 + A_0\d f_0)
$$
by Lemma \ref{elem_pp}. We then obtained the estimates $(\ref{Low}), (\ref{High})$ for the new error $A_1$ in terms of certain low and high order norms of the previous error $A_0$. 

We would like to repeat the above procedure on $D_1$ to further reduce 
the new
error. However, in order to define the approximate solution $F_1 = I + f_1$ on $D_1$ via the homotopy formula, we have to show that $D_1$ is still a strictly pseudoconvex domain with $C^2$ boundary. This is true provided that the initial error is small enough. In fact, we shall set up an iteration scheme and prove a general statement in the next proposition.

Without loss of generality, we assume that 
$$
  D_0 = \{z\in \U : \rho_0(z) < 0\} \subset \ov{D_0}\subset \U \subset B_0
$$ 
where $\rho_0$ is some $C^2$ defining function of $D_0$, $\U$ is some open neighborhood of $D_0$  and $B_0 = \{z\in \C^n : |z| < 100\}$.

Next, we discuss how the Levi-form of a $C^2$ domain is  controlled by a sequence of
$C^2$ diffeomorphism.

It will be convenient to extend the defining function of  a domain to a larger and fixed domain. Let $\rho_0$ be a $C^m$ defining function of $D_0$ on $\mathcal U$. Suppose that   $\pd \U\in C^1$ and  $D_0$ is relatively compact in $\mathcal U$.
Define
\eq{tEu}
\widetilde Eu=\chi E_{\mathcal U}u+(1-\chi),
\eeq
where $\chi\geq0$ is a smooth  function that equals $1$ on $\mathcal U_1$ for some $\mathcal U_1$ and  has compact support in $B_0$ and $E_{\mathcal U}\rho_0>0$ on $\ov{\mathcal U_1}$, and furthermore
$\ov{\mathcal U}\subset\mathcal U_1\subset B_0$. 
\le{iter_cont}
Fix a positive integer $m$.  Let $D_0\subset \mathcal U\subset B_0 \subset \rr^N$ with $\ov{D_0}\subset \mathcal U$. 
Suppose that $D_0$ admits a $C^m$ defining function $\rho_0$ satisfying
$$
D_0=\{x\in \mathcal U\colon \rho_0(x)<0\}
$$
where $\rho_0>0$ on $\ov {\mathcal U}\setminus D_0$ and $\nabla\rho_0\neq0$ on $\pd D_0$. 
Let 
$F_j = I + f_j$ be a $C^m$ diffeomorphism which maps ${B_0}$ onto $B_0$ and maps $D_j$ onto $D_{j+1}$. 
Let 
$\rho_1=(\widetilde E\rho_0)\circ F_0^{-1}$ and $\rho_{j+1}=\rho_j\circ F_j^{-1}$ for $j>0$, which are defined on $B_0$.
For any $\e>0$, there exists 
$$\delta=\delta(\rho_0,\epsilon, m)>0$$
such that if
\eq{}
\norm{f_j}_{{B_0},m}\leq \frac{\delta}{(j+1)^2},\quad 0\leq j< L,
\eeq
then
we have the following.
\bppp
\item
 $\tilde F_j=F_j\circ\cdots\circ F_0$ and $\rho_{j+1}$ satisfy
\ga\label{C2conv-g}
\norm{\tilde F_{j+1}-\tilde F_j}_{{B_0},m}\leq C_m\frac{\delta}{(j+1)^2},\quad 0\leq j< L\\
\norm{\tilde F_{j+1}^{-1}-\tilde F_j^{-1}}_{{B_0},m}\leq C'_m\frac{\delta}{(j+1)^2},\quad 0\leq j< L,\\
\norm{\rho_{j+1} - \rho_0}_{{\U},m} \leq \e, \quad 0\leq j< L.
\label{Ucont}
\end{gather}
\item
 All $D_j$ are contained in $\mathcal U$ and
\eq{distDj}
\dist (\pd D_j,\pd D)\leq C\e, \quad \dist(D_j,\pd {\mathcal U})\geq\dist(D_0,\pd {\mathcal U})-C\e.
\eeq
\eppp
In particular, when $L=\infty$, $\widetilde F_j$ converges in $C^m$ to a $C^m$ diffeomorphism from ${B_0}$ onto itself, while $\rho_j$ converges in $C^m$  of ${B_0}$ as $\widetilde F^{-1}_j$ converges in $C^m$ norm on the set.
\ele
\begin{proof} 
Let us denote $\tilde E\rho_0$ from \eqref{tEu} by $\rho_0$. 

$(i)$
 Let $\tilde F_i=I+\tilde f_i$.  We have
  $
 \tilde f_{i+1}=f_{i+1}\circ \tilde F_i+\tilde f_i.
 $
 By the chain rule, we get
 $$
 \norm{\tilde f_{i+1}-\tilde f_i}_{{B_0},m}\leq C_m\norm{f_{i+1}}_{m}\prod_i(1+\norm{f_i}_m)^{2m}\leq C_m'\norm{f_{i+1}}_m\leq C_m\frac{\delta}{(i+1)^2}.$$ Let $F_i^{-1}=I+g_i$ and $\tilde F_i^{-1}=I+\tilde g_i$. 
On $B_0$,  we can use the identity  $\tilde F_{i+1}^{-1}-\tilde F_i^{-1}=\tilde F_{i+1}^{-1}-F_{i+1}\circ \tilde F_{i+1}^{-1}$.  Thus
 $$
 \tilde g_{i+1}-\tilde g_i=- f_{i+1}\circ  \tilde G_{i+1}.
$$
By the chain rule, we get $\norm{G_{i+1}}_m\leq C_m(1+\norm{\tilde f_{i+1}}_m)^{2m}\leq C_m'$. Then we obtain
$$
\norm{
\tilde g_{i+1}-\tilde g_i}_m\leq C_m\norm{f_{i+1}}_m, \quad 
\norm{\tilde g_i}_m\leq C_m\delta.$$

So far we have not used any assumption on $\delta$ other than the condition that $\delta<C$. To verify \eqref{Ucont}, we must use the uniform continuity of the $m$-th derivatives of $\rho_0$.  Let $D_K$ be a derivative of order $k$. 
We have
\ga
\rho_{i+1}-\rho_0=\rho_0\circ \tilde G_i-\rho_0 \\
 D_K(\rho_{i+1}-\rho_0)=(D_K\rho_0)\circ \tilde G_{i}-D_K\rho_0+\sum P_{K,K'}(\pd_x \tilde g_{i}, \dots, \pd _x^k\tilde g_{i})D_{K'}\rho_0 \label{de_cond}
\end{gather}
where $P_{K,K'}(\pd_x\tilde  g_{i}, \dots, \pd _x^k\tilde g_{i})$ is a polynomial without constant term.  Thus its sup norm is bounded by
$$
C_m\norm{\tilde g_{i}}_m\leq C_m'\delta.
$$ 
Applying the chain rule, we bound the sup norm of $D_{K'}\rho_i$ by $C\norm{\rho_0}_m$. By the uniform continuity of $D_K\rho_0$ and the estimate
$$
\norm{
\tilde g_{i}}_0\leq C_m\delta,
$$
we therefore obtain $\norm{(D_K\rho_0)\circ \tilde G_{i}-D_K\rho_0}_0<\e/2$ when $\delta$ is sufficiently small.  

$(ii)$
Applying \eqref{Ucont} for $m=1$ implies that when $\e<\e(\rho_0)$ and $\e(\rho_0)>0$  is sufficiently small,  we have
\eqref{distDj}, where $C$ depends only on $\nabla \rho_0$. \end{proof} 

To use the Levi-form of $\rho$ at $z$, it will be convenient to define
$$
T_z^{1,0}\rho=\left\{\sum t_j\pd_{ z_j}\colon \sum t_j\pd_{z_j}\rho(z)=0\right\}, \quad |\sum t_j\pd_{z_j}|=\sqrt{\sum|t_j|^2}.
$$

\le{levi-forms}
 Let $D$ be a  relatively compact $C^2$ domain in $\mathcal U$  defined by a $C^2$ functions $\rho$. 
There are 
$\e=\e(\rho)>0$ and a neighborhood $\mathcal{N}=\mathcal{N}(\rho)$ of $\pd D$ such that if 
 $\norm{\tilde\rho-\rho}_{\mathcal U,2}<\e$, then we have
 $$
\inf_{\tilde z, \tilde t, |\tilde t|=1}\{ L\tilde\rho(\tilde z,\tilde t)\colon
 \tilde t\in T_{F(z)}^{1,0}\tilde\rho, \tilde z\in \mathcal{N}\}
\geq\inf_{z, t,|t|=1} \{L\rho(z,t)\colon  t\in T_z^{1,0}\rho,  z\in \mathcal{N}\}-C
\e.
 $$
  Furthermore, $\tilde D=\{z\in \mathcal U\colon\tilde\rho<0\}$ is a $C^2$ domain with
  $\pd \widetilde D\subset\mathcal N(\rho)$.
\ele 
\begin{proof} 
Since $\rho$ is a $C^2$ defining function of $D$, there exists a neighborhood $\mathcal{N}$ of $\d D$ such that $\nabla \rho(z) \ne 0$ for $z\in\mathcal{N}$. 
Let $\tilde z  \in \mathcal{N}$. Without loss of generality, we assume that $\tilde z = 0$ and $\frac{\d  \rho}{\d z_n} \ne 0$.  When $\delta$ is small, we still have $\frac{\d \tilde \rho}{\d z_n} \ne 0$.  Consequently, we know that 
$T^{1,0}_0 \tilde \rho$ is spanned by $\{\d_i - \frac{\d_i \tilde\rho(0)}{\d_n \tilde\rho(0)}\d_n\}_{i=1}^n$. Let 
$$
\tilde t_i = \f{\d_i - \frac{\d_i \tilde\rho(0)}{\d_n \tilde\rho(0)}\d_n}{|\d_i - \frac{\d_i \tilde\rho(0)}{\d_n \tilde\rho(0)}\d_n|}, \quad t_i = \f{\d_i - \frac{\d_i \rho(0)}{\d_n \rho(0)}\d_n}{|\d_i - \frac{\d_i \rho(0)}{\d_n \rho(0)}\d_n|}.
$$
Then we have
\begin{align*}
\left|\sum_{i,j}\f{\d^2\tilde\rho(0)}{\d z_i \d \ov{z_j}} \tilde t_i \ov{\tilde t_j} - \sum_{i,j}\f{\d^2\rho(0)}{\d z_i \d \ov{z_j}}  t_i \ov{ t_j} \right|  \leq  C\e.
\end{align*}
Shrinking $\mathcal N$ if necessary, we have $\rho>\e'$ on $\pd \mathcal N\setminus D$. Taking $\e<\e'/2$, we conclude that $\pd \widetilde D$ is contained in $\mathcal N$. 
\end{proof}

Before stating our main result in this section, let us fix some notations that will appear in the next proposition. Let
$\U$ be the same open neighborhood of $D_0$ appeared in the homotopy formula \eqref{htf} and
$$
\quad B_0 = \{z\in \C^n: |z| < 100\}.
$$
Recall that the constants $C_2^*, C_{s}^*, C_{s}^{**}, C_r^*, C_m^*$ appeared in \eqref{initial_est}, \eqref{initial_est1}, \eqref{ini_condition2}, \eqref{Low}, \eqref{High} and \eqref{High_inf} have been fixed.  Next, recall that for a bounded strictly pseudoconvex domain $D_0$ with a $C^2$ {\emph{standard}} defining function $\rho_0$,  there is a positive $\e(D_0)$ such that if $\norm{\rho-\rho_0}_2<\e(D_0)$, then all the bounds in the estimates for Stein extension, Nash-Moser smoothing operator, and the homotopy operator in~\ci{MR3961327} are {\emph{upper stable}} for domains with defining function $\rho_0$. See Remark~\ref{stremark}. Finally, let  
\begin{equation}\label{delD0}
\de 
(\rho_0,\e)
= \de(
\rho_0, \e,
2), \quad \de(\rho_0)=\de(\rho_0,\e(D_0),2).
\end{equation}
 be the constant from Lemma \ref{iter_cont}.
 \pr{iteration}
Let $2<s<3$ and 
$s+\sqrt2+\f{3}{2}<r
 <\infty$. Let $C_2^*$, $C_s^{**}$, $ C_r^*, \e(D_0)$,  $\de
 (\rho_0)$ be the constants stated above and let positive numbers $\a,\b,d, \g, \k$ satisfy
\ga\label{r-s}
r-s-\f{1}{2}-\g-\k>\a d+\b,\quad \a(d-1)>\f{1}{2} 
+ \k,
\quad \b(2-d)>\f{1}{2} 
+ \k.
\end{gather}
Let $ \dbar:=(\d_{\ov 1}, \dots,\d_{\ov n})^t$ be the standard complex structure on $\C^n$  and let $\d$ be its conjugate.
Let $D_0$ be a bounded strictly pseudoconvex domain 
with a $C^2$ defining function $\rho_0$ on $\mathcal U$
and $X_{(0)} = \dbar + A_0 \d\in \Lambda^r(\ov{D_0})$  be a formally integrable almost complex structure. 
There exists a constant 
$$
\hat t_0:=\hat t_0(r,s,\a,\b,d,\k, C_2^*,  C_{s}^{**}, C_r^*, \e(D_0), \de
(\rho_0)) \in (0,1/2)
$$
such that if  $0<t_0\leq \hat t_0$ and
\ga
\label{first-init}
\znorm{A_0}_{
\ov{D_0},
s}\leq t_0^\a, \\
\label{missed}
\znorm{A_0}_{\ov{D_0}, r} \leq t_0^{-\g},
\end{gather}
then the following statements are true for  $i =0,1,2\dots.$
\bppp
\item There exists a $C^\infty$ diffeomorphism   $F_i=I+f_i$ from $B_0$ onto itself  with $F_i^{-1}=I+g_i$ such that $f_i, g_i$ satisfy 
\ga
\label{f-is}
\znorm{g_i}_{
B_0,s+1}\leq C_s\znorm{f_i}_{B_0,s+1},\quad
\znorm{f_i}_{B_0,s+1}\leq C_{s}^*t_i^{-1/2}a_i
\end{gather}
where
$$
 t_{i+1} = t_{i}^d,\quad i\geq0.
$$
\item  Set $\rho_{i+1}=\rho_i\circ F_i^{-1}$. Then 
$D_{i+1} := F_i(D_i) = \{z\in \mathcal U: \rho_{i+1}< 0\}$ and 
\ga\label{rhoi+1}
\norm{\rho_{i+1} - \rho_0}_{\U, 2} \leq \e(D_0),\\
\dist( D_{i+1}, \pd \U)\geq \dist(D_0, \pd \mathcal U)-C\e.
\label{distDi+1}
\end{gather}
\item $( F_{
i}|_{\ov D_i})_\ast(X_{(i)})$ is in the span of $X_{
i+1
} := \dbar + A_{i+1} \d$ on $\ov{D_{i+1}}$.  Moreover,  $ a_{i+1}= \znorm{A_{i+1}}_{D_{i+1},s}$ and $ L_{i+1} = \znorm{A_{i+1}}_{D_{i+1}, r}$ satisfy
$$
a_{i+1} \leq t_{i+1}^\a,\quad L_{i+1}\leq L_0t_{i+1}^{-\b}.
$$ 
\item If in addition $A_0\in C^\infty(\ov D_0)$, then for any $m>1$ and $M_i = \znorm{A_i}_{D_i, m}$,  we have
$$
\znorm{f_i}_{D_i,m+\frac{1}{2}}\leq C_m M_i.
$$
Moreover, there exist some $\eta(d)> 0$ independent of $m$ and $N = N(m, d) \in \N$ such that for all $i > N$, 
\eq{MMN}
M_{i}\leq M_{N} t_i^{-\eta}.
\eeq
\eppp
\end{prop}

The 
 assertions $(i),(ii)$
 above clearly imply that $D_i$ is a strictly pseudoconvex domain with $C^2$ boundary.  According to the remark at the end of Section \ref{sec Settings}, the assertion 
 $(ii)$
implies that we can choose the constants $C_2^*, C_s^*, C_s^{**}, C_r^*, C_m^*$ to be independent of $D_i$ provided that $\e=\e(D)$ is sufficiently small. This is important for our iteration to converge.  The last two assertions roughly say that we have rapid decay in the low order norm while rapid growth in the high order norm.  This is different from previous work, e.g. \cites{MR995504, MR3961327}, where the high order norm grows only linearly. We refer the reader to   \cites{MR995504, MR3961327} for precise definition of rapid and linear growth. Here we would like to point out that for this reason, the parameters $\a,\b,d, \g, \k, s$ will be carefully chosen in the end to both accommodate the constraints obtained in the iteration procedure and 
 achieve optimal regularity results for the convergence.

\begin{proof}
We are given $\a>0,\b>0,d>1, \g>0, \k>0$ satisfying \eqref{r-s}. We will see at the end of the proof that such $\a,\b,d, \g, \k$ exist when
$r-s>\sqrt2+3/2$. It is also clear that $\a>1/2$, $\b>1/2$ and $1<d<2$.

For the moment, we 
require
 $\hat t_0\in (0,\frac{1}{2})$.  We will further adjust $\hat t_0$ a few times 
 and indicate its explicit dependency on parameters mentioned in the statement of the proposition. This will be used in next section to prove the lower stability of $\de_r(D_0)$ which appeared in  Theorem \ref{Global NN}.

\medskip
We consider first the case when $i = 0$. 

Let $E_{
0
} $ be the Stein extension operator on $D_{0}$ and let $S_{t_{0}} : C^0(\U)\to C^\infty (D_{0})$ be the Nash-Moser smoothing operator. Let $P_{D_{0}, \U}, Q_{D_{0}, \U}$ be the homotopy operators defined in Section  \ref{sec Settings} (we shall abbreviate them as $P_{0}, Q_{0}$ for simplicity). We defined in Section \ref{approximation_of_embedding} that
$$
f_{0}= -S_{t_{0}}E_{0} P_{0}  A_{0} , \quad F_{0} = I + f_{0}.
$$ 
By an abuse of notation, we still denote by $F_{0} $  its extension $E_{0} F_{0}$ to $C^\infty(\C^n)$.   

 Assume that \eqref{ini_condition2} holds, i.e.
 \eq{copy1}
t_{0}^{-1/2} |A_{0}|_{D_0, s}  \leq \frac{1}{C_s^{**}}. 
\eeq
Then according to \eqref{estgss} and
\eqref{initial_est1}, we have
\begin{gather}
\label{copyestgss}
\znorm{g_0}_{
{B_0}, 
s+1}\leq C_s \znorm{f_0}_{{B_0}, s+1},\\
\label{copyinitial_est1}
\znorm{f_0}_{{B_0}, s+1}\leq
C^{*}_{s} t_0^{-1/2}\znorm{A_0}_{D_0,s}.
\end{gather}
To achieve \eqref{copy1}, we require that
\eq{t0_1}
\hat t_0 \leq\left(\frac{1}{C_s^{**}}\right)^{\frac{2}{2\a-1}}.
\eeq
Then it is clear that \eqref{copy1} is ensured by \eqref{first-init} and \eqref{t0_1} since  $\a>1/2$ is fixed and 
$$
t_{0}^{-1/2} |A_{0}|_{D_0, s} \leq t_0^{\a - 1/2} \leq \frac{1}{C_s^{**}},
$$
when $t_0\leq\hat t_0$. This gives us $(i)$ for $f_0,g_0$ without requiring \eqref{missed}.

Now we verify $(ii)$ when $i=0$.  Let $\de = 
\de(\rho_0) 
$ be the constant appeared in \eqref{delD0}.
Assume that 
\eq{t0_2}
\hat t_0 \leq \left(\frac{\de}{C_2^*}\right)^{2}.
\eeq
Then according to \eqref{initial_est} and \eqref{first-init},  for $0 < t_0 \leq \hat t_0$, we have
\eq{}
\norm{f_{0}}_{B_0, 2} \leq C_2^* a_{0} \leq C_2^* t_0^{\a }.
\eeq
Consequently, we obtain from \eqref{Ucont} 
$$
\norm{\rho_{1} - \rho_0}_{\mathcal U, 2} \leq \e(D_0).
$$
By Lemma \ref{iter_cont} $(ii)$,  we have
$$
\dist(D_1, \d U) \geq \dist(D_0, \d U) - C\e.
$$ 

To verify $(iii)$ for $i=0$, let us recall what we proved in Section \ref{four term estimate}. Since \eqref{ini_condition2} is satisfied for $i = 0$ by \eqref{t0_1},  then from \eqref{Low} and \eqref{High} we know that  $a_{1}=\znorm{A_{1}}_{D_{1}, s}, L_{1} = \znorm{A_{1}}_{D_{i+1},r}$ satisfy 
\ga\label{copy2}
a_{1}\leq C^*_r\cdot( t_0^{r-s-1/2}L_0+t_0^{-1/2}a_0^2),
\\
L_{1}\leq C^*_r\cdot (t_0^{-1/2}L_0).
\label{copy3}
\end{gather}
 We would like to show that 
\eq{induction}
a_{1}\leq t_{1}^\a,\quad\a\geq1/2;  \quad
L_{1}\leq L_0 t_{1}^{-\b}
\eeq
where $\a > 1/2,\b >0$ and $t_{1}=t_0^{d}$ for some $d>1$.

Here we must use \eqref{missed} in addition to \eqref{first-init}.    
Recall that the positive parameters $\a,\b,\k, \g$ have been given  such that
\gan
\{(\a, \b, d, s): \a d+\b < r-s-1/2-\k-\g\} \ne \O.
\end{gather*}
Next, choose $\hat t_0\in(0,1/2)$  so that  
\eq{t0_3}
\hat t_0  \leq \left(\frac{1}{2 C^*_r}\right)^{\frac{1}{\k}}.
\eeq
Note that this implies $2C^*_r t_0 ^{\k} \leq 1$ for $0<t_0<\hat t_0$.
Then it is easy to see from \eqref{copy2} and \eqref{copy3}  the following inequalities 
\gan
a_1\leq \frac{1}{2}( t_0^{r-s-1/2}t_0^{-\k-\g}+t_0^{2\a}t_0^{-1/2} t_0^{-\k})\leq t_0^{\a d}= t_{1}^\a,\\
L_{1}\leq t_0^{-1/2}
 t_0^{-\k}L_0\leq t_0^{-\b d}L_0=  t_{1}^{-\b}L_0.
\end{gather*}
Here we have used
 \eqref{first-init}-\eqref{missed} and  assumed the following constraints on $ \a, \b, d,\k, \g$
\ga
\a d < r-s-1/2-\k-\g ,\\
\a(2-d)>1/2+\k, \quad \a>0,  \\
\b d>1/2+\k, \quad \b>0. 
\end{gather}
We have now verified $(iii)$ for $i=0$ assuming the intersection of these constraints is nonempty. We will see in the induction step that this is true provided that $2 < s < 3$ and $s +\frac{3}{2}+ \sqrt{2}< r < \infty$.

Part $(iv)$ will be proved separately at the end of the proposition.
 
\medskip
Now assume that the induction hypotheses hold for some $i-1 \in \N$, $i \geq 1$. 

$(i)$  By induction hypothesis $(i)$, $(ii)$ and Lemma \ref{levi-forms}, we know that $D_{i}$ is a $C^2$ strictly pseudoconvex domain. Therefore, we can
apply the construction of the approximate solution defined in Section \ref{approximation_of_embedding} on $D_i$
$$
f_i = S_{t_i} E_i P_i A_i, \quad F_i = I + f_i,
$$
where $E_{i} $ is the Stein extension operator on $D_{i}$, $S_{t_{i}} : C^0(\U)\to C^\infty (D_{i})$ is the Nash-Moser smoothing operator and  $P_i = P_{D_{i}, \U}, Q_i = Q_{D_{i}, \U}$ are the homotopy operators defined in Section  \ref{sec Settings}.
  Moreover, by induction hypothesis $(ii)$, we can assume that $C_s^*$ is independent of $1,2,\cdots, i$.  Therefore, estimates \eqref{initial_est} and \eqref{initial_est1} hold for $f_{i}$,
\ga
\norm{f_{i}}_{B_{0}, 2} \leq C_2^* a_{i}, \\
\znorm{ f_{i}}_{B_{0}, s+1}\leq C_s^* t_i^{-1/2} a_{i}.
\end{gather}
Notice that $f_i = E_i f_i$ has compact support in $B_0$. Obviously, we have  by $(ii)_{i-1}$
$$
C_2^*  a_{i} \leq C_2^*t_i^{1/2}\leq C_2^*t_0^{1/2}\leq 1/2. 
$$
Then by Lemma \ref{Webster}, $F_i$ is a diffeomorphism from $B_0$ to itself and  $G_{i}\defeq F_{i}^{-1}$ exists on $B_{0}$. 

\medskip

$(ii)$ Let $\de = \de(\rho_0, \e, s)$ be the constant appeared in Lemma \ref{iter_cont}.
Let 
$$
D_{i+1} = \{z\in \U: \rho_{i+1}(z)< 0\},
$$ where $\rho_{i+1}(z) = \rho_{i} \circ G_{i}$.
Since $D_i$ is strictly pseudoconvex by induction, by \eqref{initial_est} we get
\eq{}
\norm{f_{i}}_{B_0, 2} \leq C_2^* a_{i} \leq C_2^* t_i^{\a },
\eeq
where we used induction hypothesis $(iv)$ for $i-1$ in the last inequality.
Notice that we have
$$
C_2^* t_i^{\a} \leq \frac{\de}{(i+1)^2},
$$
assuming that 
\eq{t0_4'}
C_2^*\hat t_0^{\frac{1}{2}d^i}\leq \frac{\delta}{(i+1)^2}.
\eeq
This has been achieved for $i = 0$ in \eqref{t0_2}.  We show how to achieve this condition for all $i$ assuming $\hat t_0$ sufficiently small. 
Indeed, assume that we have achieved \eqref{t0_4'} for $i-1$. Then
$$
C_2^* \hat t_0^{\frac{1}{2}d^{i}}=C_2^* \hat t_0^{\frac{1}{2}d^{i-1}}\hat t_0^{\frac{1}{2}d^{i-1}(d-1)}\leq \frac{\de}{i^2} \hat t_0^{\frac{1}{2}d^{i-1}(d-1)} \leq \frac{\delta}{(i+1)^2},
$$
where the last inequality holds for all $i \geq 1$ 
 by  requiring that
\eq{t0_4}
\hat t_0 \leq 4^{-\frac{2}{d-1}}.
\eeq
Consequently, we obtain from \eqref{Ucont} 
$$
\norm{\rho_{i+1} - \rho_0}_{\mathcal U, 2} \leq \e(D_0).
$$
Note that \eqref{distDi+1} follows from \eqref{rhoi+1} and  \eqref{distDj}. 

\medskip 

$(iii)$ The verification for $(iii)$ in the general case is similar to the case when $i = 0$. However,  as we will see extra constraints on the parameters $\a, \b, d, \k, \g$ appear, when $i > 0$.   

According to induction hypothesis $(iv)$ and \eqref{t0_1}, we have
$$
t_{i}^{-1/2} |A_{i}|_{D_i, s} \leq t_i^{\a - 1/2} \leq \frac{1}{C_s^{**}}. 
$$
Moreover, since $D_{i}$ is a $C^2$ strictly pseudoconvex domain, we know from Section \ref{sec correction estimate} that $(\ref{Low}), (\ref{High}) $ are valid for 
$a_{i+1}=\znorm{A_{i+1}}_{D_{i+1}, s}, L_{i+1} = \znorm{A_{i+1}}_{D_{i+1},r}$:
\gan
a_{i+1}\leq C^*_r\cdot( t_i^{r-s-1/2}L_i+t_i^{-1/2}a_i^2), \\
L_{i+1}\leq C^*_r\cdot (t_i^{-1/2}L_i).\end{gather*}
Here the constant  $C^*_r$ does not depend on 
$ i$ by induction hypothesis $(ii)$.  
Notice that 
$$
a_i \leq t_i^\a,\quad L_i \leq L_0 t_i^{-\b}, \quad 2 C_r^* t_i^\k \leq 1
$$
where the first two inequalities is nothing but induction hypothesis $(iv)$  and the last condition follows easily from \eqref{t0_3} since $t_i < t_0$. Then after a computation similar to the case when $i = 0$, we obtain
\gan
a_{i+1}\leq \frac{1}{2}( t_i^{r-s-1/2}t_i^{-\b}t_i^{-\k-\g}+t_i^{2\a}t_i^{-1/2} t_i^{-\k})\leq t_i^{\a d}= t_{i+1}^\a,\\
L_{i+1}\leq t_i^{-1/2}
 t_i^{-\k}t_i^{-\b}L_0\leq t_i^{-\b d}L_0=  t_{i+1}^{-\b}L_0,
\end{gather*}
where we have assumed
\ga
\a d  +\b < r-s-1/2-\k-\g , \label{constraint_1}\\
\a(2-d)>1/2+\k, \quad \a>0,   \label{constraint_2}\\
\b(d-1)>1/2+\k, \quad \b>0.   \label{constraint_3}
\end{gather}
Note that the first and third constraints are more restrictive than the ones for $i = 0$.

\medskip

Before proceeding to the proof of $(iv)$, we briefly discuss how the parameters $\a,\b,\g$ etc can be chosen to satsify condition \eqref{r-s}. 

Let $\xi = r-s-1/2$. Let $\mathcal{D}(\xi,d,\k,\g) \subset \R^2$ be the set of $(\a,\b)$ such that  \eqref{constraint_1}, \eqref{constraint_2} and \eqref{constraint_3} are satisfied. We must determine the values of $\xi,d, \k, \g$ so that $\mathcal{D}(\xi,d, \k, \g)$ 
is nonempty.

We readily notice that  $1<d<2, \a>1/2+\k, \b>1/2+\k$ and $r> \frac{7}{2}$ since $s > 2$. 
We consider the limiting domain for fixed $\xi, d$ and  $\k= \g = 0$
\[
\mathcal{D}(\xi,d, 0, 0) =\left \{(\a,\b)\in \R^2 :  \a d+\b < \xi, \,\, \a(2-d)>\frac{1}{2}, \,\, \b(d-1)>\frac{1}{2}\right\}.
\]

We first determine condition on $\xi,d$ so that $ \mathcal D(\xi,d, 0, 0)$ is non empty. By the defining equations of $ \mathcal D(\xi,d, 0, 0)$,  it is non empty if and only if
\eq{}
\xi >p(d), \quad p(d):=\frac{d }{2(2-d )}+\frac{1}{2(d -1)}, \quad 1 < d < 2.
\eeq
Note that on interval $(1,2)$, $p$ is a strictly convex function which attains minimum value $p(\sqrt2)=\sqrt2+1$. This implies that
$$
r-s-\frac{1}{2} > p(\sqrt{2}) = \sqrt{2} + 1.
$$
Therefore, we obtain the minimum smoothness requirement for our complex structure
$$
r > s + 
\f{3}{2} + \sqrt{2} > 
\frac{7}{2} + \sqrt{2}.
$$
Notice that $\mathcal{D} (\xi, d, \k, \g)$ is still nonempty for sufficiently small $\k, \g$.
Consequently, we have found a set of values $\a,\b, d, \k, \g$ so that the constraints are satisfied. However, we remark here that our goal is to obtain the convergence of $\znorm{A_j}_{D_j, \ell}$ where $s\leq \ell \leq r$ for as large $\ell$ as possible. To achieve this, we need to optimize our choice of the constants $\a,\b, d, \k, \g$ together with $s$. This will be done in the next section.

\medskip 

$(iv)$ The case when $A \in C^\infty$ needs an additional estimate.  We still keep all previous assumptions. In particular, $r, s$ are fixed finite numbers.  Thus we have $(i), (ii), (iii)$. 
Recall that in Section \ref{four term estimate}, we constructed the smoothing operator $S_t$ by choosing $L = 3$ in \eqref{smoothing_op} if $A\in C^\infty$.  Let $M_{i} \defeq \znorm{A_{i}}_{m}$.

Since $D_{i}$ is strictly pseudoconvex, it follows from \eqref{correctionest} that 
$$
\znorm{f_{i}}_{m+\frac{1}{2}} \leq C_m M_{i}.
$$
Since \eqref{ini_condition2} holds for $f_i$,  it follows from \eqref{High_inf} that
$$
M_{i+1}\leq C^*_m\cdot (t_i^{-1/2}M_i).
$$

We would like to show \eqref{MiM0}, i.e., there exist some $\eta = \eta(d)$ and $N = N(m,d)$ such that for all $i > N$, we have
$$
M_i \leq M_N t_i^{-\eta}.
$$
Note that this holds trivially for $i=N$. 

Let $N = N(m, d)\in\N$ be sufficiently large so that for all $i > N$,
$$
 C^*_m t_i^{\l} \leq 1.
$$
Then we have an estimate which is almost identical to the estimate of $L_{i+1}$ 
for $i > N(m,d)$,
\ga
\label{MiM0}
M_{i+1} \leq t_i^{-\l} t_i^{-1/2} M_i\leq  t_i^{-\l-1/2-\eta}M_N \leq t_i^{-\eta d} M_N = t_{i+1}^{-\eta} M_N 
\end{gather}
where we have fixed an $\eta$ satisfying
$$ 
\eta(d - 1) > 1/2 + \l.
\qquad
\qedhere
$$ 
\end{proof}
 
\section{Optimal Regularity and Convergence of iteration}\label{convergence}
\setcounter{thm}{0}\setcounter{equation}{0}

Let us first explain what we mean by optimal regularity. Here we will use the interpolation  methods in Moser [19] and Webster [27]. 

Let $2 < s < 3$ and $s + \frac{3}{2} + \sqrt{2} < r < \infty$.  Assume that the given initial integrable almost complex structure $X_0 = \dbar + A_0 \d$ is in $\Lambda^r(\ov{D_0})$. 
Moreover, we assume that conditions \eqref{first-init} and \eqref{missed} from Proposition \ref{iteration} are satisfied. That is, 
\eq{copy8}
\znorm{A_0}_{\ov{D_0},s}\leq t_0^\a, \quad \znorm{A_0}_{\ov{D_0}, r} \leq t_0^{-\g} 
\eeq
for some $\a, \g, t_0$ satisfying the requirements in Proposition \ref{iteration}.

By convexity of H\"older-Zygmund norms and \eqref{induction}, we can control the intermediate derivatives $\ell=(1-\th)s+\th r$ for $0< \th < 1$,
\eq{conv}
\znorm{A_{j+1}}_{D_{j+1}, \ell}\leq C_r\znorm{A_{j+1}}_{D_{j+1}, s}^{1-\th}\znorm{A_{j+1}}_{D_{j+1}, r}^\th\leq C_r t_{j+1}^{(1-\th)\a-\th\b}
\eeq
where $j\in \Z^+$. To achieve the convergence of $\znorm{A_{j+1}}_{D_{j+1}, \ell}$, we need $(1-\th)\a>\th\b$. Therefore, $0\leq \th <\frac{\a}{\a+\b}<1$. Let $\theta_0= \frac{\a}{\a+\b}$ and we would like to maximize
\eq{optimize}
\ell(\a,\b,r,s,d)=s+\theta_0(r-s)= s+\frac{\a}{\a+\b} (r-s)
\eeq
under the constraints (\ref{constraint_1}), (\ref{constraint_2}) and (\ref{constraint_3}). Notice that we cannot achieve the maximum.

Set  $\xi= r-s-1/2$. Recall from the proof of induction hypothesis $(iv)$ in Proposition \ref{iteration} the following facts.
We have defined
 $\mathcal{D}(\xi,d, \k, \g) \subset \R^2$ to be the set of $(\a,\b)$ such that  \eqref{constraint_1},  \eqref{constraint_2} and \eqref{constraint_3} are satisfied. 
That is, 
$$
\mathcal{D}(\xi,d, \k, \g) = \{(\a,\b)\in \R^2 :  \a d+\b+\k+\g < \xi, \,\, \a(2-d)>\frac{1}{2}+\k, \,\, \b(d-1)>\frac{1}{2}+\k\}.
$$
We consider the limit domain when $\k = \g = 0$,
\[
\mathcal{D}(\xi,d, 0, 0) = \{(\a,\b)\in \R^2 :  \a d+\b < \xi, \,\, \a(2-d)>\frac{1}{2}, \,\, \b(d-1)>\frac{1}{2}\}
\]
and $$
\quad p(d):=\frac{d }{2(2-d )}+\frac{1}{2(d -1)}, \quad 1 < d < 2,
$$
which is a strictly convex function on $(1, 2)$
and attains minimum value $p(\sqrt2)=\sqrt2+1$. It is clear that $\mathcal{D}(\xi, d, 0, 0)$ is nonempty if and only if 
\eq{condition}
\xi >p(d).
\eeq
It is also easy to see that the closure of $\mathcal{D}(\xi,d, 0, 0)$ is  \[
\overline{
\mathcal{D}(\xi,d, 0, 0)} = \left\{(\a,\b)\in \R^2 : \a d+\b \leq \xi,\,\, \a(2-d)\geq \frac{1}{2}, \,\, \b(d-1)\geq \frac{1}{2} \right\}.
\]
We write  \[
\ell(\a,\b,r,s,d)=r-\widetilde \ell(\a,\b,\xi,d), \quad \widetilde \ell(\a,\b,\xi,d) \defeq\frac{\xi+\frac{1}{2}}{\frac{\a}{\b}+1}.
\]
On $\overline{\mathcal D(\xi, d, 0, 0)}$, let us minimize
$$\widetilde \ell(\a,\b,\xi,d)= \frac{\xi+\frac{1}{2}}{\frac{\a}{\b}+1}. $$
Since $\overline{\mathcal D(\xi, d, 0, 0)}$ is a compact set, the continuous function  $\widetilde \ell$ achieves minimum at some point $(\a_\infty,\b_\infty)\in \overline{\mathcal D(\xi, d, 0, 0)}$. It is clear that $\b_\infty$ takes the smallest possible value and $\a_\infty$ takes the largest possible value. Thus,
$$
\b_\infty=\frac{1}{2(d-1)}, \quad
\a_\infty=\frac{\xi-\b_\infty}{d}.
$$
Thus, we have
\begin{align}
\widetilde \ell(\a_\infty,\b_\infty,\xi,d)&=\frac{\xi +\frac{1}{2}}{
2\frac{d-1}{d}(\xi-\b_\infty)+1}\\
& = \frac{d(\xi + \frac{1}{2})}{2(d-1)\xi + (d-1)} \\
&= \frac{1}{2}(1+\frac{1}{d-1}).
\nonumber
\end{align}

Fix any $r > 3 + \sqrt{2}$ and $\xi = r - s - \frac{1}{2}$.  In order to achieve the optimal regularity, we would like to find the largest $d$ that satisfies \eqref{condition}. Since $p(d)$ is a strictly convex function, the largest $d$ for a given $\xi$ is therefore achieved by the larger solution $d(\xi)$ where $p(d(\xi)) = \xi$.  By a simple computation, we have \eq{formula_d}
d(\xi)= \frac{1+3\xi + \sqrt{\xi^2-2\xi-1}}{1+2\xi} = 1 + \frac{1}{2 + \frac{1}{\xi}} + \frac{\sqrt{1-\frac{2}{\xi}-\frac{1}{\xi^2}}}{2 + \frac{1}{\xi}}.
\eeq
It follows that $d(\xi)$ is an increasing function in $\xi$. In particular, it approaches  to $2^-$
and $\widetilde \ell(\a_\infty,\b_\infty,\xi,d)$ tends to $1$ as $\xi$ and hence $r$ tend to $+\infty$. 

We observe that for any given $r > 3 + \sqrt{2}$,
$d(\xi)$ is maximized when $\xi=r-s-1/2$ takes the 
maximum value 
$$
\xi_\infty \defeq r - \frac{5}{2}
$$  
for $s=2$. Let $d_\infty = d(\xi_\infty)$. Then $\widetilde \ell$ achieves its minimum 
$\widetilde\ell_\infty$
at $(\a_\infty,\b_\infty,\xi_\infty,d_\infty)\in \mathcal{D}(0, 0)$ where $\mathcal{D}(0, 0) = \cup_{(\xi, d)} D(\xi,d, 0, 0)$ for all $(\xi,d)$ satisfies the condition \eqref{condition}.

In particular when
$$r= 5,$$ a simple computation shows that
$$ \xi_\infty = 5 - 5/2= 5/2, \quad d_\infty = d(\xi_\infty)= 3/2,
\quad \widetilde\ell_\infty=3/2.$$
By monotonicity of $d(\xi)$, this implies that, when 
 $r-3> s>2$ we can take $\xi=r-s_*$ for a suitable $2<s_*<s$,  the above values then satisfy 
 $$  
d_\infty\geq 3/2 +c_0, \quad  \widetilde \ell_\infty
 \leq 3/2 - \widetilde c_0
$$
for sufficiently small $c_0, \widetilde c_0>0$.
 
In summary, we have proved the following:  
Let $\mathcal{D}(\k, \g) = \cup_{\xi, d}\mathcal{D}(\xi, d, \k, \g) \subset \R^4$ be the set of $(\a,\b,\xi, d)$ satisfying \eqref{constraint_1}, \eqref{constraint_2} and \eqref{constraint_3}. Let  $r > 5$. We choose $2<s_*
<
s
$ such that 
$$
\xi_* = r - s_* - 1/2 > 5/2.
$$ Consequently,
$$
d(\xi_*) > 3/2.
$$
Then we choose $3/2 < d_* <  d(\xi_*)$ so that $\xi_* > p(d_*)$. This ensures that $\mathcal{D}(\xi_*, d_*, 0, 0)$ is nonempty.
Therefore, there exist parameters $\alpha_*, \beta_*, \xi_*, d_*$ so that  $\mathcal{D}(0, 0)$ is nonempty.  Notice that $\k, \g$ can be arbitrarily close to $0$. Then for sufficiently small $\k, \g$, we have $(\a_*, \b_*,\xi_*,d_*) \in \mathcal{D}(\k, \g)$. Moreover, since $d_* > 3/2$, 
\eq{const4_t_0}
\widetilde \ell(\a_*, \b_*,\xi_*,d_*) = \frac{1}{2}(1+\frac{1}{d_*-1})
\leq 3/2 - 2c^*
\eeq
for sufficiently small $c^* > 0$.  It is clear that $c^*$  depends only on the choice of $\xi_*, d_*$ specified above. Note that \eqref{const4_t_0} implies
$$
\ell_* = \ell(\a_*, \b_*,\xi_*,d_*) 
\geq r-\frac{3}{2} + 2c^*.
$$
Let $\ell = 
s_* + \th(r - s_*)$ and choose $0\leq \th < \frac{\a_*}{\a_* + \b_*}$ such that
$$
\ell > \ell_*-c^{*}. 
$$
Suppose that \eqref{copy8} is satisfied for the above choices of $\a_*$ and  $\g$. Then 
\eq{weaker}
\znorm{A_0}_{D_0,s_*}\leq \znorm{A_0}_{D_0, s} \leq t_0^{\a_*}, \quad \znorm{A_0}_{D_0,r}\leq t_0^{-\g}.
\eeq
Consequently, we know from \eqref{conv} that
\eq{Al_conv}
\znorm{A_j}_{D_j,\ell} \leq 
C_r t_j^{a_\ell}, \quad a_\ell  > 0
\eeq  
where $a_\ell = (1-\th)\a_* - \th\b_*$ for 
the above $\ell$ that  satisfies
\eq{room}
\ell > r-\frac{3}{2} + c^*.
\eeq

Finally, we are ready to show the convergence of the sequence $\widetilde F_j = F_{j-1}\circ\cdots\circ F_0$  to some embedding $F$ on $\ov{D_0}$ in $\Lambda^{r-1}(\ov{D_0})$ for any $r>5$ (including $r=\infty$).  Moreover, $F$ maps the perturbed almost complex structure to the standard one and $F (D_0) \defeq D$ is still a $C^2$ strictly pseudoconvex domain in $\C^n$.

\pr{ConvMapping}
Let $5 < r < \infty$ and $2 < s < r - 3$.  Let $D_0$ be a $C^2$ strictly pseudoconvex domain in $\C^n$ and $X_{\overline\a} = \d_{\overline\a} + A_{\overline\a}^\b \d_{\b} \in \Lambda^r(\overline{D_0})$, $\a=1,\cdots, n$ be a formally integrable almost complex structure on $D_0$. There exist constants $\a > 1/2$, 
$\g \in (0, 1)$ and $\hat t_0 \in (0, 1/2)$ such that if
\ga\label{cond8.1}
\znorm{A}_{D_0,s}\leq t_0^\a \quad\text{and}\,\,\quad  \znorm{A}_{D_0, r}\leq t_0^{-\g},
\end{gather}
where $0 < t_0 \leq \hat t_0$, then the following statements are true.
\bppp
\item There is a sequence of mappings $\widetilde F_j$ converge to some embedding $F :  \ov{D_0}\to \C^n$ in $\Lambda^{\ell+\frac{1}{2}}(\ov{D_0})$  for any $0\leq \ell\leq r-3/2 + c^*$. Here $c^{*}>0$ is the same constant appeared in \eqref{room}. In particular,  $F\in C^{r-1}(\ov{D_0})$. 
\item  If in addition $A \in C^\infty(\ov{D_0})$, then $F \in C^\infty(\ov {D_0})$
under \rea{cond8.1} and the weaker condition  $r-3/2-\sqrt2>s>2$. 
\item   $F_*(X_{\ov\a})$ are in the span of $\d_{\ov 1}, \dots, \d_{\ov n}$ and $F(\ov{D_0})$ is strictly pseudoconvex. 
\item 
 The
$
\de_r(D_0):=\hat t_0^\a
$
is lower stable under a small $C^2$ perturbation of $\d D_0$.

\eppp
\epr
\begin{proof}
We may assume that $2<s<3$.

$(i)$  Let us first determine the constants $\a, \g$ and $\hat t_0$. Recall $\hat t_0$ from Proposition \ref{iteration} where
$$
\hat t_0:=\hat t_0(r,s,\a,\b,d,\k, C_2^*,  C_{s}^{**}, C_r^*, \e(D_0), \de
(\rho_0)) \in (0,1/2).
$$
Notice that $r, C_2^*,  C_{s}^{**}, C_r^*, \e(D_0), \de
(\rho_0)$ have been specified before the proof of Proposition~\ref{iteration}.  Choose $\k, \g$ and $(\a,\b,d, \xi)\in \mathcal{D}(\k, \g)$ such that \eqref{const4_t_0} is satisfied.  
Then $\hat t_0$ is determined by the constraints \eqref{t0_1}, \eqref{t0_2}, \eqref{t0_3} and \eqref{t0_4} appeared in 
Proposition~\ref{iteration}.   These constraints will be written down explicitly when proving the stability of $\de_r(D)$ in $(iv)$.

By assumption, for $0 < t_0 \leq \hat t_0$, we have
\eq{}
 \znorm{A_0}_{D_0,s}\leq t_0^\a, \quad \znorm{A_0}_{D_0, r} \leq t_0^{-\g}.
\eeq
Consequently, Proposition \ref{iteration} is now valid for such choices of $\hat t_0$ and $A_0$. Moreover, 
$$
\znorm{f_j}_{D_j, \ell+\frac{1}{2}}\leq C_\ell |A_j|_{D_j, \ell} \leq C_\ell t_j^{a_\ell},\quad a_\ell > 0
$$ 
according to \eqref{correctionest} and \eqref{Al_conv}. 

Consider the composition $\widetilde{F}_{j+1} = F_j \circ F_{j-1}\circ \cdots \circ F_0$ where $F_j = I + f_j$ for $j\geq 0$.
Let $\ell = r-3/2 + c^*$. We use Lemma \ref{chain_rule} to estimate 
\begin{align}\label{cauchy_seq}
\znorm{\widetilde{F}_{j+1}-\widetilde{F}_j}_{D_0, \ell+\frac{1}{2}} &= \znorm{f_{j} \circ F_{j-1}\circ\cdots\circ F_0}_{D_0, \ell + \f{1}{2}}
\\&\nonumber\leq  (C_\ell)^j\left\{\znorm{f_j}_{\ell + \f{1}{2}}+\sum_i \left(\norm{f_j}_2\znorm{f_i}_{\ell + \f{1}{2}}+\znorm{f_j}_{\ell+\f{1}{2}}\znorm{f_i}_s \right)\right\}
\\&\nonumber\leq (C_\ell)^j  C \znorm{f_j}_{\ell+\f{1}{2}}\leq 
C_r^j t_j^{a_\ell},\quad \text{for some $a_\ell >0$.}
\end{align}
This shows that $\znorm{\widetilde{F}_{j+1}-\widetilde{F}_j}_{D_0,\ell+\frac{1}{2}}$ is a 
Cauchy sequence since $\sum_j C_\ell^jt_j^{a_\ell}$ clearly converges. We denote the limit mapping by $F$.

$(ii)$ The case $r = \infty$  needs a separate argument because the construction of smoothing operator $S_t$ depends on the $r$ in the finite smooth case.

We are going to use \eqref{f-is}, \eqref{MMN} from Proposition \ref{iteration} and convexity \eqref{zconvexity} without the optimization process. Indeed, let 
$$
\eta(d), \quad N(m ,d) \in \N
$$
be the same constants from $(v)$ in Proposition \ref{iteration}.
Then for $\ell=(1-\th)s + \th m$,  $j > N(m, d)$, we have
\eq{smooth_conv}
\znorm{f_{j+1}}_{D_{j+1}, \ell+\frac{1}{2}}\leq C_m\znorm{f_{j+1}}_{D_{j+1}, s+\frac{1}{2}}^{1-\th}\znorm{f_{j+1}}_{D_{j+1}, m+\frac{1}{2}}^\th \leq C_m' t_{j+1}^{(1-\th)\a-\th\eta}.
\eeq
We have the convergence provided that $(1-\th)\a - \th \eta > 0$, which can be achieved by choosing any $0 < \th < \frac{\a}{\a+\eta}< 1$. For instance, we can choose $\th = \frac{\a}{2(\a+\eta)}$. 

Then we can apply the same argument \eqref{cauchy_seq} to see that $F\in \Lambda^{\ell+1/2}({\ov{D_0}})$ where 
$$
\ell + 1/2 = s + \frac{\a}{2(\a + \eta)}(m -s)+ 1/2 > r_0 - 1- 1/2
$$ for $m$ sufficiently large.
Since $m$ can be arbitrarily large and $\th$ is independent of $m$,  we conclude that $F\in\Lambda^{\ell}(\ov{D_0})$ for all $\ell$. This implies that $F\in C^\infty(\ov{D_0})$.

$(iii)$ By part $(iv)$ in Proposition \ref{iteration}, we see that $F$ transforms the formally integrable almost complex structure into the standard complex structure. By $(ii)$ in Proposition \ref{iteration}, we know that $D \defeq F(D_0)$ is a strictly pseudoconvex domain with $C^2$ boundary in $\C^n$.

Finally, we show that $F$ is a diffeomorphism. Since $F$ is
$\Lambda^{r-1}$, it suffices to check the Jacobian of $F(x)$ for $x\in D_0$.
\eq{}
\begin{aligned}
\znorm{D F - I}_{D_0, 0} \leq \sum_{j=0}^\infty \znorm{D \widetilde{F}_{j+1}-D \widetilde{F}_j}_{D_0, 0}&\leq \sum_{j=0}^\infty \znorm{\widetilde F_{j+1}-\widetilde F_j}_{D_0, 1}\leq \frac{1}{2},
\end{aligned}
\eeq
where the last inequality follows from $(\ref{cauchy_seq})$.

$(iv)$ Let $\e(D_0)$ be the size of second order perturbation of $\rho_0$ such that we have {\emph{upper stablility}} of $C_2^*(\rho_0), C_s^*(\rho_0), C_s^{**}(\rho_0), C_r^*(\rho_0)$.

Recall that $\hat t_0$ is determined by the constraints  \eqref{t0_1},   \eqref{t0_2}, \eqref{t0_3} and \eqref{t0_4}. More specifically,
\eq{t0_final}
\hat t_0 \leq \min\left\{\left(\frac{1}{C_s^{**}}\right)^{\f{2} {2\a - 1}}, \left(\frac{
\de
(\rho_0)
}{C_2^*}\right)^{\frac{1}{\a}}, \left(\frac{1}{2C^*_r}\right)^{\f{1}{\k_0}}, \left(\frac{1}{4}\right)^{\frac{2}{d-1}}  \right\}.
\eeq
Here $C_2^*(\rho_0), C_s^{**}(\rho_0), C_r^*(\rho_0)$ are upper stable constants and 
$\de  
(\rho_0)$ 
 is given by \eqref{delD0} and satisfies the properties in Lemma~\ref{iter_cont}.
See also Section~\ref{sec Settings} for details on upper stability.

Let us replace $\delta 
(\rho_0)$ by a smaller quantity $\delta^*(D_0)$ defined by
$$
\delta^*(D_0)  :=
\min\left\{ \f{\e(D_0)}{4 
C''}, 
\de\left(
\rho_0, \frac{\e(D_0)}{2}, 2\right)\right\},
$$
where 
$C''$ is an absolute constant determined later.  Then we 
have $0 < t_0 \leq \hat t_0$ for
$$
 t_0
(D_0) 
:= \min\left\{\left(\frac{1}{C_s^{**}}\right)^{\f{2} {2\a - 1}}, \left(\frac{
\de^*(D_0)
}{C_2^*}\right)^{\frac{1}{\a}}, \left(\frac{1}{2C^*_r}\right)^{\f{1}{\k}}, \left(\frac{1}{4}\right)^{\frac{2}{d-1}}  \right\}.
$$
Define $\de_r(D_0) = t_0^\a(D_0)$, that is that
\eq{de_r_final}
\de_r(D_0) 
:= \min\left\{\left(\frac{1}{C_s^{**}}\right)^{\f{2} {2\a - 1}}, \left(\frac{
\de^*(D_0)
}{C_2^*}\right)^{\frac{1}{\a}}, \left(\frac{1}{2C^*_r}\right)^{\f{1}{\k}}, \left(\frac{1}{4}\right)^{\frac{2}{d-1}}  \right\}^\a.
\eeq

Finally, we show that $\de_r(D_0)$ is {\emph{lower stable}} under $C^2$ perturbation.  
Let $(\widetilde D, \widetilde X)$ be a pair of strictly pseudoconvex domains and formally integrable complex structures that satisfy the conditions of the proposition. Let $\widetilde D = \{z\in\U: \tilde \rho<0\}$ and $\de_r(\widetilde D)$ be the corresponding stability constant to be determined. 

Recall that we say $\de_r(D_0)$ is  lower stable under $C^2$ perturbation of $\rho_0$ if the following holds.  
There exist
\eq{ls_condition}
\e^*
(\rho_0)> 0,\ C(\rho_0) > 0
\eeq
such that if $\norm{\tilde \rho-\rho_0}_{\U,2}\leq 
\e^*(\rho_0)$, then we can choose $\de_r(\widetilde D)$ satisfying
\eq{ls}
\de_r( D_0) \leq C(\rho_0) \de_r(\widetilde  D).
\eeq

We start by choosing
\eq{esrho}
\e^*(\rho_0) := 
\frac{\e(D_0)}{4 C''},
\eeq
where as mentioned above $C''>1$ is an absolute constant to be determined. 

Next, 
for the domain $\widetilde D$ we define
\eq{de_r_D}
\de_r(\widetilde D) 
:= 
\hat t_0^\a(\widetilde D)
:=\min\left\{\left(\frac{1}{C_s^{**}}\right)^{\f{2} {2\a - 1}}, \left(\frac{
\hat \de (\widetilde D)
}{C_2^*}\right)^{\frac{1}{\a}}, \left(\frac{1}{2C^*_r}\right)^{\f{1}{\k}}, \left(\frac{1}{4}\right)^{\frac{2}{d-1}}  \right\}^{\a}
\eeq
where 
$C_2^*, C_s^{**}, C^*_r$ depend on $\widetilde D$, and 
$$
\hat \delta (\widetilde D)
:= \de\left(\tilde\rho, \frac{\e(D_0)}{
2}, 2\right).
$$
Note that the second argument of the last expression  does not depend on $\widetilde D$. 

Notice by definition of upper stability, the reciprocal of a upper stable constant is lower stable. It also follows from definition that  taking 
{\it minimum}
 of lower stable constants or raising to certain fixed
 positive
 power do not change lower stability.

Therefore, in order to prove that $\de_r(D_0)$ is lower stable, it suffices to show that 
if
the initial domain $\widetilde D$ has a defining function $\tilde\rho$ 
satisfying $$\norm{\tilde\rho-\rho_0}_{\U,2}<\e^*(D_0)$$
then the following hold:
\begin{enumerate}
\item Proposition \ref{iteration} holds for the pair $(\widetilde D, \widetilde X)$ with $\e(\widetilde D)$, $\de(\tilde\rho, \e(\widetilde D), 2)$ and $\hat t_0$ being replaced by $\e(D_0) / 2$, $\hat \de(\widetilde D)= \de(\tilde\rho, \e(D_0)/2 , 2 )
$ and $\hat \de (\widetilde D)^{{1}/{\a}}$ while the rest of the statements remain unchanged.  
\item Let $\tilde\rho_1,\tilde\rho_2,\dots, $ be the sequence of defining functions for domains obtained in previous assertion 
for the initial domain $\widetilde D$ with defining function $\tilde\rho$ 
satisfying $\norm{\tilde\rho-\rho_0}_{\U,2}<\e^*(D_0)$. Then 
$$
\norm{\tilde\rho-\rho_0}_{\U,2}<\e(D_0), \quad 
\norm{\tilde\rho_j - 
\rho_0}_{\U, 2}\leq \e(
D_0).
$$
Moreover, we get an embedding for $(\widetilde D,\widetilde  X)$ with the given $\de_r(\widetilde D)$ in \eqref{de_r_D}. 
\item We can use the same set of parameters $\a, d, \k$ for initial defining functions $\rho_0,  \tilde\rho$.  
\item Finally, we have $
\de^*(D_0)\leq \hat\de(\widetilde D)$, i.e. 
\eq{star-hat-de}
\de(\rho_0, \frac{\e(D_0)}{
4}, 2) \leq  \de(\tilde\rho, \frac{\e(D_0)}{2}, 2).
\eeq
In other words, $\de(\rho_0, \frac{\e(D_0)}{
4}, 2)$ 
 fulfills the requirements for $\de(\tilde\rho, \frac{\e(D_0)}{2}, 2)$. 
Notice here the difference in domains and scale of perturbation. Clearly, \eqref{de_r_final}, \eqref{de_r_D} and \eqref{star-hat-de} imply immediately \eqref{ls}.
\end{enumerate}

These assertions follow in principle  from the proofs. However,  let us point out how to achieve them. 

To see the first assertion, we only need to argue that we can replace $\e(\widetilde D)$ by $\e(D_0) / 2$.  The rest of the changes are obvious. One way to see this is to give a precise estimate of how $\e(\widetilde D)$ depends on the defining function. However, we give an alternative argument based on the proof of Proposition \ref{iteration} itself. 

Notice that the function  $\e(\widetilde D)$,
 replacing $\e(D_0)$
 in Proposition \ref{iteration}, is two-folds. On the one hand, we need to control the Levi forms of sequence of domains. On the other hand, in order to get convergence, we need to make sure that we
 can use the same coefficients $C_2^*, C_s^*, C_s^{**}, C_r^*$ in the estimates during the iteration despite that the domains $\widetilde D_j$ are changing with $\widetilde D$.  

Let $f_j$ be the sequence of corrections in Proposition~\ref{iteration} 
for $\widetilde D$.  Then they are guaranteed to satisfy the requirements in Lemma~\ref{iter_cont} when \eqref{t0_2} and \eqref{t0_4} are satisfied. These two conditions are achieved by our choice of  $\hat t_0$. Here Lemma~\ref{iter_cont} is applied to $\tilde\rho, \hat \de(\widetilde D)$ and the sequence $f_j$.

Consequently, by Lemma~\ref{iter_cont} 
applied to $\widetilde D$ and $\e=\e(\widetilde D)$, we have \eq{small}
\norm{\tilde\rho_j - \rho_0}_{\U, 2} = \norm{\tilde\rho_j - \tilde\rho + \tilde\rho - \rho_0}_{\U, 2} \leq 
\f{\e( D_0)}{2}+\e^*(D_0)<
\e(D_0),
\eeq
provided we can verify 
\eq{tildesmall}
\norm{\tilde\rho_j - \tilde\rho}_{\U,2}<\f{\e( D_0)}{2}.
\eeq
Therefore, the sequence of domains defined by $\tilde\rho_j$ are strictly pseudoconvex provided that $\e(D_0)$ is sufficiently small.
Note also that by \eqref{small}, we have
$$
C_2^*(\tilde\rho_j) \leq C(\rho_0) C_2^*(\rho_0).
$$
Similar estimates hold for $C_s^*, C_s^{**}, C_r^*$.

Thus we have verified the second assertion. 
However, strictly speaking, one should 
check \eqref{small}
by induction in $j \in \N$ as we did in the original proof of Proposition \ref{iteration}. However, since the ideas are the same, we leave the detail to interested reader. 
  We will  verify  \eqref{tildesmall} below and  show how  $\hat\delta(\widetilde D), \delta_r(\tilde D)$ are chosen.

For the third assertion, we note that the choices of $\a , d, \k$ depend only on  the constraints \eqref{constraint_1},  \eqref{constraint_2}, \eqref{constraint_3} and the optimization process.  Therefore, they can be chosen uniformly.

Finally, we will show that  $\delta^*(D_0)\leq\hat\de(\widetilde D)$, i.e. 
$$
\de(\rho_0, \frac{\e(D_0)}{
4}, 2) 
\leq \de(\tilde\rho, \frac{\e(D_0)}{2}, 2),
$$ 
which amounts to verifying that  $\de(\rho_0, \frac{\e(D_0)}{
4}, 2)$
 fulfills the requirements for $\de(\tilde\rho, \frac{\e(D_0)}{2}, 2)$.

Indeed,  let $F_j = I + f_j$ be the sequence of diffeomorphisms that satisfy the condition of Lemma~\ref{iter_cont}
in which $D_0,\delta$ are replaced by $\widetilde D, \delta^*(
D_0)$.  Thus, we can assume
$$
\norm{f_j}_{B_0,2}\leq \frac{\delta^*( D_0)}{(j+1)^2}.
$$
Let $\tilde F_j = I + \tilde f_j = F_j \circ\cdots F_0$ and $\tilde G_j = I + \tilde g_j = F_{j}^{-1}\circ\cdots F_0^{-1}$. Let $|K| = |K'| = 2$.  Set $\rho_j=\rho\circ \tilde G_j$,  $\tilde\rho_j=\tilde\rho\circ G_j$, $\rho'=\tilde\rho-\rho$ and $\rho_j'=\rho'\circ \tilde G_j$. We have 
$$
\norm{\tilde g_j}_{\U,2} \leq C_2 \de^*(D_0)\leq C_2
$$ 
and
\aln
\norm{\tilde\rho_j-\tilde\rho}_{\U,2}&\leq \norm{\rho_j-\rho}_{\U,2}+\norm{\rho'\circ\tilde G_j}_{\U,2}+\norm{\rho'}_{\U,2}\\
&\leq 
\f{\e(D_0)}{4}+(1+C'_2\norm{\tilde g_j}_{\U,2})^2\norm{\rho'}_{\U,2}
\leq \f{\e(D_0)}{4}+C_2''\e^*(D_0)\leq\f{\e(D_0)}{2}.\end{align*}
Therefore,  we have
$$
\norm{\tilde\rho_j -\tilde \rho}_{\U, 2} \leq  \e(D_0) / 2, \quad
\norm{\tilde\rho_j-\rho_0}_{\U,2}<\e(D_0).$$
This completes the proof of  assertion $(4)$ and also assertion $(2)$. 
 
Having verified all four assertions,  we conclude that $\de_r(D_0), \delta_r(\widetilde D)$ defined by \eqref{de_r_final} and \eqref{de_r_D}, are lower stable at $D_0$ under small $C^2$ perturbation.  This completes the proof.  
\end{proof}

\newcommand{\doi}[1]{\href{http //dx.doi.org/#1}{#1}}
\newcommand{\arxiv}[1]{\href{https //arxiv.org/pdf/#1}{arXiv #1}}

\def\MR#1{\relax\ifhmode\unskip\spacefactor3000 \space\fi%
\href{http //www.ams.org/mathscinet-getitem?mr=#1}{MR#1}}

\bibliographystyle{plain}

\end{document}